\documentclass[11pt,reqno,british,a4paper,twoside]{amsart}
\usepackage[sc]{mathpazo}
\usepackage{amssymb,babel}
\usepackage{array,booktabs,mathtools,paralist,url,xspace}

\mathtoolsset{mathic=true}
\setdefaultenum{\upshape (i)}{}{}{}

\theoremstyle{plain}
\newtheorem{theorem}{Theorem}[section]
\newtheorem{proposition}[theorem]{Proposition}
\newtheorem{lemma}[theorem]{Lemma}
\newtheorem{corollary}[theorem]{Corollary}

\theoremstyle{definition}

\theoremstyle{remark}
\newtheorem{remark}[theorem]{Remark}
\newtheorem{example}[theorem]{Example}

\numberwithin{equation}{section}
\numberwithin{table}{section}

\newcolumntype{C}{>{$}c<{$}}
\newcolumntype{L}{>{$}l<{$}}

\DeclareMathOperator{\ad}{ad}

\DeclareMathOperator{\Der}{Der}
\DeclareMathOperator{\End}{End}

\DeclareMathOperator{\im}{Im}
\DeclareMathOperator{\re}{Re}

\DeclareMathOperator{\Stab}{Stab}

\newcommand{\Lie}[1]{\operatorname{\rm{#1}}}
\newcommand{\lie}[1]{\operatorname{\mathfrak{#1}}}
\newcommand{\G}{\Lie{G}}
\newcommand{\LH}{\Lie{H}}
\newcommand{\LK}{\Lie{K}}
\newcommand{\GL}{\Lie{GL}}
\newcommand{\SO}{\Lie{SO}}
\newcommand{\SL}{\Lie{SL}}
\newcommand{\Sp}{\Lie{Sp}}
\newcommand{\SU}{\Lie{SU}}
\newcommand{\Un}{\Lie{U}}
\newcommand{\g}{\lie{g}}
\newcommand{\h}{\lie{h}}
\newcommand{\lk}{\lie{k}}
\newcommand{\n}{\lie{n}}

\newcommand{\so}{\lie{so}}
\newcommand{\spto}{\lie{sp}(2)\oplus\lie{sp}(1)}

\newcommand{\bR}{{\mathbb R}}
\newcommand{\bZ}{{\mathbb Z}}

\newcommand{\bT}{{\mathbb T}}

\newcommand{\HKT}{\textsc{hkt}\xspace}
\newcommand{\QKT}{\textsc{qkt}\xspace}

\newcommand{\Smod}{{\mathcal S}}
\newcommand{\Wmod}{{\mathcal W}}

\newcommand{\alt}{\mathbf{a}}

\newcommand{\Hodge}{{*}}

\newcommand{\hook}{{\lrcorner\,}}

\newcommand{\LC}{\nabla}

\newcommand{\w}{\wedge}

\DeclarePairedDelimiter{\Span}{\langle}{\rangle}
\newcommand{\abs}[1]{\left\lvert#1\right\rvert}

\newcommand{\br}{\hspace{0pt}}

\begin{document}

\title[Harmonic structures \& intrinsic torsion]{Harmonic structures and intrinsic torsion} 

\author{Diego Conti}
\address[D.~Conti]{ Dipartimento di Matematica e Applicazioni\\ Universit\`a di Milano
  Bicocca\\ Via Cozzi 55\\ 20125 Milano\\ Italy}
\email{diego.conti@unimib.it}

\author{Thomas Bruun Madsen}
\address[T.\,B.~Madsen]{Centro di Ricerca Matematica Ennio De Giorgi\\ Scuola Normale Superiore\\ 
Collegio Puteano\\ Piazza dei Cavalieri 3\\  56100 PISA\\ Italy}
\email{thomas.madsen@sns.it}

\thanks{The authors are grateful to Simon Salamon for generously sharing ideas and showing his interest in this project.
TBM thankfully acknowledges financial support from the \textsc{Danish Council for Independent Research, Natural Sciences}.}

\begin{abstract}
   We discuss the construction of \br\( \Sp(2)\Sp(1) \)-structures whose fundamental form is closed.
  In particular, we find \( 10 \) new examples of \( 8 \)-dimensional nilmanifolds that admit an invariant closed \( 4 \)-form 
  with stabiliser \( \Sp(2)\Sp(1) \). Our constructions entail the notion of \( \SO(4) \)-structures
  on \( 7 \)-manifolds. We present a thorough investigation of the intrinsic torsion of such structures, leading to the construction of 
  explicit Lie group examples with invariant intrinsic torsion.
\end{abstract}


\maketitle

\section{Introduction}
\label{sec:intro}

A quaternionic K\"ahler structure on a manifold \( M \) of dimension \( 4n\geqslant8 \) is an \( \Sp(n)\Sp(1) \)-structure preserved by the Levi-Civita connection; equivalently, the \( 4 \)-form \( \Omega \) associated to the structure is required to be parallel (cf. \cite{Gray:ANoteOn}). This condition makes \( M \) a quaternionic Hermitian manifold, and forces the underlying metric to be Einstein (see \cite{Salamon:RiemannianGeometryAnd}). 

\cite{Swann:AspectsSymplectiquesDe} was the first to realise that if the dimension of \( M \) is at least \( 12 \), then \( \Omega \) cannot be closed unless it is parallel. However, his arguments also showed that in dimension \( 8 \), the condition \( d\Omega=0 \) fails to fully determine the covariant derivative, thereby leaving open the possible existence of \emph{harmonic \( \Sp(2)\Sp(1) \)-structures}, that is \( 8 \)-manifolds admitting an \( \Sp(2)\Sp(1) \)-form which is closed but not parallel. This problem can be viewed as a quaternionic analogue of the search for strictly almost K\"ahler manifolds in almost Hermitian geometry.  

This use of the adjective ``harmonic'', consistent with \cite[Definition 26]{Witt:Specialmetricsand}, refers to harmonicity of \( \Omega \): as \( \Omega \) is self-dual, the condition \( d\Omega = 0 \)  implies
\( d^\Hodge\Omega = 0 \), which means \( \Omega \) is a harmonic \( 4 \)-form.

Quaternionic K\"ahler manifolds are very rigid objects. In fact, there are no non-symmetric compact quaternionic K\"ahler \( 8 \)-manifolds of positive scalar curvature, and a conjecture by
\cite{LeBrun-S:Strongrigidityof} asserts that the same result holds in dimensions \( 4n\geqslant12 \).  This scarceness of examples motivates the study of almost quaternionic
Hermitian geometries which are, in some sense, close to being quaternionic K\"ahler. In dimension \( 8 \), harmonic \( \Sp(2)\Sp(1) \)-structures constitute one such class.

It was \cite{Salamon:Almostparallelstructures} who first constructed an example of a harmonic\linebreak \( \Sp(2)\Sp(1) \)-structure. The basic idea, which was developed further by \cite{Giovannini:SpecialStructuresAnd} in his PhD thesis (see also \cite{Catellani:TriLagrangianmanifolds}), uses the identification of \( \Sp(2)\Sp(1)\cap\SO(6) \)
with \( \SO(3) \) acting diagonally on \( \bR^3\oplus\bR^3 \); one constructs the fundamental form \( \Omega \) on \( \bT^2\times N^6 \) from a triple of forms, namely a pair of \( 3 \)-forms \( \tilde\alpha,\tilde\beta\in\Omega^3(N) \) and a non-degenerate \( 2 \)-form \( \tau\in\Omega^2(N) \). In addition to the condition \(\Stab(\tilde{\alpha},\tilde{\beta},\tau)\cong\SO(3) \),
one imposes the closedness of these three forms so as to ensure \( d\Omega=0 \). Salamon also posed the idea that one can generalise the construction by starting from the subgroup  \( \Sp(2)\Sp(1)\cap\SO(7)\cong\SO(4) \). We confirm that this is the case: then the fundamental form \( \Omega \) on \( S^1\times N^7 \) arises from a \( 3 \)-form \( \alpha \) and a \( 4 \)-form \( \beta \) on \( N \) whose stabilisers are different copies of the non-compact exceptional Lie group \( \G_2^* \) with intersection isomorphic to \( \SO(4) \). Closedness of \( \Omega \) is enforced by requiring that both of the forms \( \alpha \) and \( \beta \) are closed. This condition, \( d\alpha=0=d\beta \), can be rephrased in terms of the intrinsic torsion \( \xi\in T^*N\otimes\so(4)^\perp \); it forces \( \xi \) to take its values in a \( 49 \)-dimensional submodule. As \( \Hodge\alpha=\beta \), we have dubbed such \( \SO(4) \)-structures as \emph{harmonic}.

Whilst our construction recovers one of the known examples of a harmonic \( \Sp(2)\Sp(1) \)-structure, it also enables us to find several new examples. More precisely, we show that if \( N_i \), \( 1\leqslant i\leqslant11 \), is a nilpotent Lie group associated with the algebra \( \mathfrak n_i \) in Table \ref{table:harmonicliealg},
then  \( N_i \) admits a left-invariant harmonic \( \SO(4) \)-structure, hence \( \LK_i =\bR\times N_i\) admits a harmonic \( \Sp(2)\Sp(1) \)-form. By choosing a lattice \( \Gamma_i\subset \LK_i \), we obtain \( 11 \) examples of compact
harmonic \( \Sp(2)\Sp(1) \)-manifolds. Except for the case corresponding to \( (0,0,0,0,0,0,12,13) \), these examples are new. 

\begin{theorem}
\label{thm:main}
Each of the nilmanifolds \( M_i=\Gamma_i \backslash  \LK_i \), \( 1\leqslant i\leqslant 11 \), admits a harmonic \( \Sp(2)\Sp(1) \)-structure.
\end{theorem}

As these nilmanifolds are not Einstein (cf. \cite{Milnor:Curvaturesofleft}), \( \Omega \) cannot be parallel. In fact, for any  \( \Sp(2)\Sp(1) \)-structure on \( \bR\times N \) induced by an \( \SO(4) \)-structure on \( N \),  \( \Omega \) can be parallel only if the associated metric is hyper-K\"ahler. In this case, the underlying \( \SO(4) \)-structure is torsion-free, i.e., \( \xi=0 \). In the compact setting this implies that \( N \) is finitely covered (via a local isometry) by a flat torus or the product of a flat torus with a K3 surface (cf. Theorem \ref{thm:intrzero}).

\begin{table}[tp]
  \centering
  \begin{tabular}{LCC}
    \toprule
    \n_1        &      (0,0,0,12,23,-13,26-34-16+25)  &    (3,7,10)  \\ 
    \midrule
    \n_2        &      (0,0,0,0,0,12,13)                      &    (5,13,21) \\
    \midrule
    \n_3        &      (0,0,0,0,12-34,13+24,14)               &    (4,11,16) \\
    \midrule
    \n_4        &      (0,0,0,0,12+34,23,24)                  &    (4,11,17) \\
    \midrule
    \n_5        &      (0,0,0,0,12,13,34)                   &    (4,11,16) \\
    \midrule
    \n_6        &      (0,0,0,12,13,23+14,25+34)              &    (3,6,10)  \\
    \midrule
    \n_7        &      (0,0,0,12,13,15+35,25+34)              &    (3,6,10)  \\
    \midrule
    \n_8        &      (0,0,0,0,0,12,34)                      &    (5,12,18) \\
    \midrule
    \n_9        &      (0,0,0,0,0,12,14+23)                   &    (5,12,18) \\
    \midrule
    \n_{10}     &      (0,0,0,0,0,13+42,14+23)                &    (5,12,18) \\
    \midrule
    \n_{11}    &       (0,0,0,0,12,13,14+25)                  &    (4,9,13)  \\
    \bottomrule
  \end{tabular}
   \caption{The Lie algebras \( \lk_i=\bR\oplus\n_i \), with \( \n_i \) listed above, all admit a harmonic \( \Sp(2)\Sp(1) \)-structure. The rightmost column gives the Betti numbers
     \( b_1(\n_i),b_2(\n_i),b_3(\n_i) \). The notation should be understood as follows. The dual \( \n_2^* \) of the Lie algebra \( \n_2 \)  has a basis \( n^1,\ldots,n^7 \)
     such that \( dn^1=0=\cdots=dn^5 \), \( dn^6=n^1\w n^2 \) and \( dn^7=n^1\w n^3 \).}
   \label{table:harmonicliealg}
\end{table}

Harmonic \( \Sp(2)\Sp(1) \)-structures appear to be fairly rigid objects. Indeed, case-by-case computations for the above nilmanifold examples show that, at the infinitesimal level, the \( \SO(8) \) orbit of the \( \Sp(2)\Sp(1) \)-structure contains no other harmonic structures than those obtained by Lie algebra automorphisms. 

On a \( 7 \)-dimensional nilpotent Lie algebra, there is a very useful obstruction to the existence of a calibrated \( \G_2 \)-structure \cite[Lemma 3]{Conti-F:Nilmanifoldswithcalibrated}. As the forms \( \alpha \) and \( \beta \) are very similar to the calibrated \( \G_2 \)-form, in that they have open \( \GL(7,\bR) \)-orbit and are calibrated, one might be tempted to look for similar obstructions in our setting so as to classify nilpotent Lie algebras admitting a harmonic \( \SO(4) \)-structures. This, however, turns out to be less tractable than expected. The main complication originates from the fact that the form \( v\mapsto (v\hook\alpha)\w(v\hook\alpha)\w \alpha \), similarly for \( \beta \), has split signature. In addition to being closed and non-degenerate, the forms \( \alpha \) and \( \beta \) are required to satisfy certain compatibility conditions. This is reminiscent of the case of half-flat \( \SU(3) \)-structures, defined by a pair \( (\psi,\sigma) \) of closed forms (cf. \cite{
Hitchin:Stableformsand}). Nilmanifolds with a half-flat structure were classified in \cite{Conti:Halfflatnilmanifolds}, using a cohomological obstruction based on the fact that a simple \( 2 \)-form \( \gamma \) satisfying \( \gamma\w\psi = 0 =\gamma \w\sigma \) is necessarily zero. However, this  way of expressing the compatibility conditions is not appropriate in the case of \( \SO(4) \), because the equation \( \gamma\w\alpha=0 \) alone implies \( \gamma=0 \). In summary, a classification of harmonic \( \SO(4) \)-structures on nilmanifolds remains an open problem.

As already mentioned, compact \( 7 \)-manifolds with an \( \SO(4) \)-structure whose holonomy reduces are very rare. It is therefore natural to consider generalisations of the torsion-free condition, e.g., the case of invariant intrinsic torsion. In our case, this means \( \xi \) lies in a trivial \( \SO(4) \)-submodule of \( T^*N\otimes\so(4)^\perp \); this condition is complementary to the harmonic condition in the sense that an \( \SO(4) \)-structure which satisfies both criteria is torsion-free. Whilst being relatively unexplored, the invariant intrinsic torsion setting is already known to include interesting geometries such as nearly-K\"ahler manifolds and nearly parallel \( \G_2 \)-manifolds (see also \cite{Chiossi-M:SO3StructuresOn}). As it is the case for these latter two types of geometries, we show (Proposition \ref{prop:consttor}) that any \( \SO(4) \)-structure with invariant intrinsic torsion, in fact, has constant torsion, i.e., \( \xi \) is independent of \( n\in N \). 

In the final part of the paper, we explain the principle behind a potential classification of invariant intrinsic torsion \( \SO(4) \)-structures on Lie groups, and use this approach to provide a number of examples. The basic idea is that for a  Lie group \( \LH \), any left-invariant \( \SO(4) \)-structure is obtained by specifying an adapted coframe \( e^1,e^2,e^3,e^4,w^1,w^2,w^3 \) of the Lie algebra \( \h \). This structure has invariant intrinsic torsion if and only if the torsion \( \tau \) of the flat connection on \( \LH \) lies in \( 2\bR\oplus\partial((\bR^7)^*\otimes\so(4)) \), where \( \partial \) denotes the canonical alternating map (see \eqref{eqn:alternating}). In order to obtain concrete examples, we pick an adapted coframe and make a suitable ansatz for the form of \( \tau \). Subsequently, we check that the equations specified by \( \tau \) actually determine a Lie algebra structure.

For the groups \( \LH_i \), corresponding to the Lie algebras appearing in Table \ref{table:invtorliealg}, we also describe the type of the associated product \( \Sp(2)\Sp(1) \)-structure on  \( S^1\times \LH_i \). On each of these Lie groups, the same adapted frame that determines the \( \Sp(2)\Sp(1) \)-structure also determines a reduction to \( \Sp(2) \). In particular, \( S^1\times\LH_i \) is an almost hyper-Hermitian manifold, meaning Riemannian with a compatible almost hypercomplex structure. Whilst many of the examples are, in fact, (integrable) hyper-Hermitian, only two of them are hyper-K\"ahler (hence, necessarily, flat).

In summary, our analysis of invariant intrinsic torsion \( \SO(4) \)-structures leads to:

\begin{theorem}
\label{thm:invintrtor}
Let \( \LH \) be a Lie group whose Lie algebra \( \h \) belongs to the list in Table \ref{table:invtorliealg}. The left-invariant \( \SO(4) \)-structure on \( \LH \) determined by the coframe \( e^1,e^2,e^3,e^4 \), \( w^1,w^2,w^3\) of \( \h^* \) has invariant intrinsic torsion. 

If \( \LH \) corresponds to either \( \h_1 \) or \( \h_5^0=\h_6^0 \), then the associated product structure on \( S^1\times \LH \) is flat and hyper-K\"ahler. In the remaining cases, the non-zero components of the intrinsic torsion are \( EH \) and \( KH \),  meaning these are quaternionic Hermitian structures which are not quaternionic K\"ahler.
\end{theorem}

\begin{table}[tp]
  \centering
  \begin{tabular}{LC}
    \toprule
    \h_1             & de^i=\sum_{j,s}(e_i\hook \varpi_j)\w w^s,\, dw^s=0                                  \\
    \midrule
    \h_2             & de^i=0,\, dw^s=w_s\hook w^{123}                                                     \\
    \midrule
    \h_3             & de^i=0,\, dw^s=\omega_s                                                              \\
    \midrule
    \h_4             & de^i=\sum_s(e_i\hook\omega_s)\w w^s,\, dw^s=-2w_s\hook w^{123}                        \\
    \midrule
    \h_5^a	     & de^1=0=de^2, de^3=e^{14}, de^4=-e^{13},                                               \\
                     & dw^1=a\omega_1, dw^2=a\omega_2+e^{1}\wedge w^3,	dw^3=a\omega_3-e^{1}\wedge w^2       \\
    \midrule
    \h_6^a	     & de^1=0=de^2, de^3=e^{14},de^4=-e^{13},                                                \\
                     &dw^1=aw^{23}, dw^2=aw^{31}+e^{1}\w w^3, dw^3=aw^{12}-e^{1}\w w^2                       \\
    \midrule
    \h_7^{a,\kappa}    & de^1=0,de^2=e^{12},de^3=e^{13}+\kappa e^{14}, de^4=e^{14}-\kappa e^{13},               \\
                     &dw^s = a\omega_s-\tfrac1aw_s\hook w^{123}+(\kappa e^1-e^2) \wedge w_s\hook w^{23}      \\    
                     & \hspace{2mm}-e^3\wedge w_s\hook w^{31}-e^4\wedge w_s\hook w^{12}\\
    \midrule
    \h_8^{a,\kappa}  & de^1=0,\, de^2=e^{12}+\kappa e^{13},\, de^3=-\kappa e^{12}+e^{13},\, de^4=2\varpi_3,                  \\
                     &dw^s = a\omega_s - \tfrac2a w_s\hook w^{123} +(\kappa e^1-e^4)\wedge w_s\hook w^{12}\\
                     & \hspace{2mm} -2e^3\wedge w_s\hook w^{31} - 2e^2\wedge w_s\hook w^{23}\\
    \bottomrule
  \end{tabular}
  \caption{The Lie algebras \( \h_i \), listed above, all admit an \( \SO(4) \)-structure with invariant intrinsic torsion. In each case, \( e^1,\ldots,e^4,w^1,w^2,w^3 \) is an adapted coframe, and the \( 2 \)-forms \( \omega_s\), \(\varpi_s \) are defined in \eqref{eqn:omegavarpi}. Parameters appearing as denominators are implicitly assumed not to be zero.}
   \label{table:invtorliealg}
\end{table}

\section{Dimensional reduction}
\label{sec:SO4str}

The key ingredient in our construction of harmonic \( \Sp(2)\Sp(1) \)-structures is a ``dimensional reduction'' in the sense that we break down the \( 8 \)-dimensional
geometry into \( 7+1 \) dimensions and identify a suitable subgroup 
\begin{equation*} 
  \SO(4)\cong\SO(7)\cap\Sp(2)\Sp(1) 
\end{equation*} 
that characterises the \( 7 \)-dimensional building blocks. 

In terms of representation theory, it is useful to keep in mind the isomorphism 
\begin{equation*} 
  \SO(4)\cong \Sp(1)\times_{\mathbb Z\slash2\mathbb Z}\Sp(1)=\Sp(1)\Sp(1). 
\end{equation*}
Then it is clear that the real irreducible representations of \( \SO(4) \)
take the form
\begin{equation*} 
  \Smod^{p,q}:=[S^pH_+\otimes S^qH_-],
\end{equation*}
where \( p+q \) is an even number, \( \dim\Smod^{p,q} =(p+1)(q+1) \), and the square brackets, \([V]\), are used to indicate that we are considering the real vector space underlying
\( V \). 

The \( 7 \)-dimensional representation of \( \SO(4) \) can now be expressed as 
\begin{equation*} 
  T=\bR^7\cong\bR^4\oplus\Lambda^2_-=\Smod^{1,1}\oplus \Smod^{2,0}. 
\end{equation*}

For later computations, it is worth recalling (see, e.g., \cite[Theorem~1.3]{Salamon:TopicsInFour}) that the tensor product of \( \SO(4) \) representations can be worked out via the formula
\begin{equation*} 
  \Smod^{p,q}\otimes \Smod^{r,s}\cong\bigoplus_{m,n=0}^{\min(p,r),\min(q,s)}\Smod^{p+r-2m,q+s-2n}. 
\end{equation*}

Let \( e_1,\cdots,e_7 \) denote the standard basis of \(
\bR^7 \) and \( e^1,\ldots,e^7 \) the dual basis of
\( T^* \). Via the inclusion \( \SO(4)\subset\SO(7) \) we obtain splittings 
\begin{equation*}
T^* \cong(\bR^4)^*\times(\bR^3)^*=\Lambda^{1,0}\oplus\Lambda^{0,1}, \quad \Lambda^kT^*\cong \bigoplus_{p+q=k}\Lambda^{p,q}.
\end{equation*}
In particular, we can naturally identify \( \Span{e^1,e^2,e^3,e^4} \) with \( (\mathbb R^4)^* \subset T^*\), and so forth. Consider now the pair of forms \( (\alpha,\beta)\in\Lambda^3 T^*\times\Lambda^4 T^* \) given by
\begin{equation} 
  \label{eq:defforms}
  \begin{gathered}
    \alpha= \omega_1 \w w^1+\omega_2 \w w^2+\omega_3\w w^3-3w^{123},\\
    \beta=-\omega_1\w w^{23}-\omega_2\w w^{31}-\omega_3\w
     w^{12}-3e^{1234},
   \end{gathered}
\end{equation}
where we have fixed a basis of the space \(\Lambda^2(\bR^4)^*=
\Lambda^2_- \oplus \Lambda^2_+\) of the form
\begin{equation}
\label{eqn:omegavarpi}
\begin{split}
  \omega_1&=e^{12}-e^{34},\omega_2=e^{13}-e^{42},\omega_3=e^{14}-e^{23},   \\
\varpi_1&=e^{12}+e^{34},\varpi_2=e^{13}+e^{42},\varpi_3=e^{14}+e^{23}, 
\end{split}
\end{equation}
whilst \( w^s=e^{4+s}\), \( s=1,2,3 \), span the annihilator of \( \bR^4\subset\bR^4\times\bR^3 \). 

The forms \( \alpha \) and \( \beta \) are stable in the sense that
the associated \( \GL(7,\bR) \)-orbits in \( \Lambda^* T^* \) are
open. Straightforward computations show:

\begin{lemma}
  \label{lem:stabg2*}
  The stabilisers of \( \alpha \) and \( \beta \) in \( \lie{gl}(7,\bR) \)
  are two different copies of \( \mathfrak{g}_2^* \) that intersect in
  a copy of \( \so(4) \).
\end{lemma}
\begin{proof}
We have the well-known decomposition \( \End T\cong\bR\oplus S^2_0T\oplus\Lambda^2T \), where
\begin{equation*}
  S^2_0T\cong\bR\oplus\mathcal{S}^{1,1}\oplus\Smod^{2,2}\oplus\Smod^{0,4}\oplus\Smod^{1,3},\, \Lambda^2T\cong\so(4)\oplus\Smod^{0,2}\oplus\Smod^{1,1}\oplus\Smod^{1,3}.
\end{equation*}

The stabilisers of \( \alpha \) and \( \beta \) are computed with respect to the action that identifies \( \Lambda^2T^* \) with \( \so(7) \)
by letting  \( \Lambda^2T^* \) act on \( T \) as 
\begin{equation}
\label{eq:soaction}
   \Lambda^2T^*\times T \ni (\omega,v)\mapsto v\hook\omega, 
\end{equation}   
meaning \( e^{12} \) is identified with \( e^1\otimes e_2- e^2\otimes e_1 \), and so forth.

Straightforward computations, using \eqref{eq:soaction}, now show that the copy 
\begin{equation}
  \label{eq:so4stab}
  \so(4)\cong\Span{\varpi_1,\varpi_2,\varpi_3}\oplus\Span{\omega_1-2w^{23},\omega_2-2w^{31},\omega_3-2w^{12}}
\end{equation}
annihilates \( \alpha \) and \( \beta \).

In addition to this copy of \( \so(4) \), each of the forms \( \alpha \) and \( \beta \) is annihilated by a copy of the module \( \Smod^{1,3} \); neither of these (different) copies of \( \Smod^{1,3} \) is contained in \( \so(7) \). In particular, we conclude that each form has stabiliser \( \g_2^* \) rather than \( \g_2 \).
\end{proof}

\begin{remark}
Starting from an \( \SO(4) \)-structure \( (\alpha,\beta) \), with associated metric \( g=g^{\mathcal H}+g^{\mathcal V} \), we may use the splitting of \( T\) into \( \bR^4\times\bR^3 \) 
so as to obtain a family of metrics. Up to conformal transformation, this family is obtained by rescaling in the fibre directions \( \mathcal V \), meaning
\( w^s\mapsto\lambda w^s=:\tilde{w}^s\),
where \( \lambda>0 \). The metric \( \tilde g \) and defining forms, \( (\tilde\alpha,\tilde\beta) \), associated with this rescaling are:
\begin{equation*}
\begin{gathered}
\tilde g=g^{\mathcal H}+\lambda^2g^{\mathcal V},\,
   \tilde\alpha= \lambda(\omega_1 \w w^1+\omega_2 \w w^2+\omega_3\w w^3)-3\lambda^3w^{123},\\
    \tilde\beta=\lambda^2(-\omega_1\w w^{23}-\omega_2\w w^{31}-\omega_3\w
    w^{12})-3e^{1234}.
\end{gathered}
\end{equation*}
\end{remark}

From the pair \( (\alpha,\beta) \) we construct a \( 4 \)-form \( \Omega \) on \( \bR^8\cong\Span{e_8}\oplus T \) as follows:
\begin{equation}
    \label{eq:quat4form}
    \Omega=\alpha\w e^8+\beta=\tfrac12(\sigma_1^2+\sigma_2^2+\sigma_3^2),
\end{equation}
where
\begin{equation}
\label{eq:sp2forms}
\sigma_1=\omega_1+e^{58}-e^{67},\,\sigma_2=\omega_2+e^{68}-e^{75},\,\sigma_3=\omega_3+e^{78}-e^{56}.
\end{equation}
It is well known \cite[Lemma 9.1]{Salamon:RiemannianGeometryAnd} that \( \Omega \) has stabiliser \( \Sp(2)\Sp(1) \) in  \( \GL(8,\bR) \).

Conversely, consider a manifold \( M \) endowed with an \( \Sp(2)\Sp(1) \)-structure whose associated fundamental form is given by
\( 2\Omega=\sigma_1^2+\sigma_2^2+\sigma_3^2 \). Let \( \iota\colon\, N\to M \) be an oriented hypersurface, with induced metric \( g \), and suppose the unit normal direction is given by \( n=-e_8 \). Then
\begin{equation*} 
  \beta=\iota^*\Omega=-(e^{12}-e^{34})\w e^{67}-(e^{13}-e^{42})\w e^{75}-(e^{14}-e^{23})\w e^{56}-3e^{1234}, 
\end{equation*}
and its Hodge star is

\begin{equation*} 
  \alpha=\iota^*(n\hook\Omega)= (e^{12}-e^{34})\w e^5+(e^{13}-e^{42})\w e^6+(e^{14}-e^{23})\w e^7-3e^{567}. 
\end{equation*}

Consequently, \( N \) inherits an \( \SO(4) \)-structure which is determined by fixing any two of \( \alpha \), \( \beta \), and \( g \). 

As already mentioned, the \( 7 \)-dimensional representation of \( \SO(4) \) can be written as \( \bR^4\oplus\Lambda^2_-=\Smod^{1,1}\oplus\Smod^{2,0} \), and in the proof of Lemma \ref{lem:stabg2*}, we used this
fact to write: 
\begin{equation*}
  \Lambda^2 T\cong2\Smod^{0,2} \oplus \Smod^{2,0}\oplus\Smod^{1,1}\oplus\Smod^{1,3}\cong \Lambda^5 T.
\end{equation*}

Another straightforward computation shows that
\begin{equation}
  \label{eq:ext-dec}
  \Lambda^3 T\cong2\bR\oplus\Smod^{0,2}\oplus\Smod^{0,4}\oplus 2\Smod^{1,1}\oplus\Smod^{1,3}\oplus\Smod^{2,2}\cong \Lambda^4T. 
\end{equation}
In particular, the space of invariant forms \( \left(\Lambda^*T^*\right)^{\SO(4)} \) is \( 4 \)-dimensional. Indeed, the decomposable forms 
\(\upsilon=e^{1234} \) and \( \Hodge\upsilon=w^{123} \) are clearly invariant.

The pair of forms 
\( \alpha+4\Hodge\upsilon, \beta+4\upsilon\in \left(\Lambda^*T^*\right)^{\SO(4)} \)
are distinguished by having stabiliser the same copy of \( \g_2\subset\so(7) \); this
is an alternative way to phrase the fact that given any \( \SO(4) \)-structure, 
the inclusion \( \SO(4)\subset \G_2 \) determines a natural \( \G_2 \)-structure.

On a \(7\)-manifold with an \(\SO(4)\)-structure, defined as above, \( \SO(7) \) acts on the pair \( (\alpha,\beta) \) with stabiliser  \( \SO(4) \). Differentiation therefore gives a map 
\begin{equation*} 
  \so(7) \to\Lambda^3T^*\oplus\Lambda^4T^*,\, A\mapsto A\cdot\alpha+A\cdot\beta 
\end{equation*}
whose kernel is \( \so(4) \). In particular, we have an inclusion of the orthogonal complement of \( \so(4) \) in \( \so(7) \):
\begin{equation*} 
  \so(4)^\perp\hookrightarrow\Lambda^3T^*\oplus\Lambda^4T^*,\, A\mapsto A\cdot \alpha+ A\cdot\beta.
\end{equation*}

The intrinsic torsion \( \xi \) of an \( \SO(4) \)-structure takes values in the \( 105 \)-dimensional space
\begin{equation*}
T^*\otimes \so(4)^\perp \cong 2\bR\oplus 3\Smod^{0,2}\oplus 2\Smod^{0,4} \oplus 3\Smod^{1,1}\oplus 3\Smod^{1,3}\oplus\Smod^{1,5} \oplus\Smod^{2,0}\oplus 2\Smod^{2,2} \oplus\Smod^{2,4},
\end{equation*}
and, due to the above inclusion, the map 
\begin{equation}
\label{eq:intriso}
 \xi\mapsto\LC\alpha+\LC\beta=\xi\cdot\alpha+\xi\cdot\beta 
\end{equation}
 is an isomorphism.

If we let \( \mathbf a \) denote the alternation map \( T^*\otimes\Lambda^rT^*\to\Lambda^{r+1}T^* \), we have the well-known relation \( d=\alt\circ\LC \).  Taking into account the isomorphism \eqref{eq:intriso}, one expects that the exterior derivatives of \( \alpha \) 
and \( \beta \) could encode valuable information about \( \xi \). Indeed, a computation shows that \( (d\alpha,d\beta) \) encodes the maximal possible amount of
information in the following sense.
 
\begin{lemma}
The map
\begin{equation}
  \label{eqn:intto}
  T^*\otimes \so(7)\to\Lambda^4T^*\oplus \Lambda^5T^*, \, e^i\otimes A\mapsto e^i\w A\cdot\alpha + e^i\w A\cdot\beta  
\end{equation}
is surjective.
\end{lemma}
\begin{proof}
Set \( \tilde T:=T\oplus\Span{e_8} \), and write \( \Omega=\alpha\w e^8+\beta \). The map appearing in the statement  can be identified with the restriction to \( T^*\otimes \so(7) \) of
\begin{equation*} 
  \tilde T^*\otimes \so(8)\to \Lambda^5\tilde T^*, \, e^i\otimes A\mapsto e^i\w A\cdot\Omega.
\end{equation*}

By \cite{Swann:AspectsSymplectiquesDe}, this map is  surjective with kernel \( KS^3H\oplus\tilde T^*\otimes(\spto) \),
and the component \( KS^3H \) can be identified with the kernel  of the skew-symmetrisation map 
\(\alt\colon \tilde T^*\otimes (\spto)^\perp \to  \Lambda^3\tilde T^* \).
It therefore suffices to prove that the composition
\begin{equation*}
  T^*\otimes\so(7)\xrightarrow{\pi} \tilde T^*\otimes (\spto)^\perp\xrightarrow{\alt} \Lambda^3\tilde T^*
\end{equation*}
is surjective. 

If we write \( \so(7) \) as 
\begin{equation*}
  \so(7)=\so(4) \oplus \so(3)\oplus\bR^4\otimes \bR^3,\, \so(3)=\Span{\omega_1+2e^{67},\omega_2+2e^{75},\omega_3+2e^{56}},
\end{equation*}
we see that the projection \( \pi \) is zero on \( T^*\otimes\so(4) \) and is the inclusion on \( T^*\otimes\so(3) \). On \( T^*\otimes \bR^4\otimes \bR^3 \), it  has the form
\begin{equation*}
  e^c\otimes e^{ab}\mapsto e^c\otimes \tfrac14(3e^{ab}-e_a\hook\sigma_s\w e_b\hook\sigma_s).
\end{equation*}

Clearly, the restriction of \( \alt\circ\pi \) to \( \Span{e^8}\otimes \bR^4\otimes \bR^3 \) is injective, with image \( e^8\w \Lambda^{1,1} \). It follows that 
\begin{equation*} 
  \alt\circ\pi\colon (\Span{e^8}\oplus\Lambda^{0,1})\otimes\bR^4\otimes \bR^3\to \Lambda^{1,2}\oplus e^8\w\Lambda^{1,1}
\end{equation*}
is surjective. Likewise, 
\begin{equation*}
  \alt\circ\pi\colon \Lambda^{1,0}\otimes\bR^4\otimes \bR^3\to \Lambda^{2,1}\oplus e^8\w \Lambda^{2,0}
\end{equation*} is an isomorphism. We conclude that  \( \alt\circ\pi \)  is surjective by considering its restrictions to  \( \Lambda^{1,0}\otimes\so(3) \) and \( (\Span{e^8}\oplus\Lambda^{0,1})\otimes \so(3)\).  
\end{proof}

As a consequence of the above observation, \( d\alpha \) and \( d\beta \) can only both be zero if \( \xi \) takes values in the modules
\begin{equation*}
  \Smod^{0,4}\oplus\mathcal{S}^{1,3}\oplus\Smod^{1,5}\oplus\Smod^{2,2}\oplus\Smod^{2,4}.
\end{equation*}

In summary:

\begin{proposition}
The conditions \( d\alpha=0=d\beta \) characterise \( \SO(4) \)-structures whose intrinsic torsion takes values in the \( 49 \)-dimensional submodule 
\begin{equation*}
  \Wmod= \Wmod^{0,4} \oplus\Wmod^{1,3}\oplus\Wmod^{1,5} \oplus\Wmod^{2,2}\oplus\Wmod^{2,4}
\end{equation*}
with each component being isomorphic to the \( \SO(4) \)-module  \( \Smod^{p,q} \) with the same indices.

\end{proposition}

We shall say an \( \SO(4) \)-structure is \emph{harmonic} if the forms \( \alpha \) and \( \beta \) are closed. This terminology reflects the fact that,
as  \( \alpha=\Hodge\beta \), the conditions \( d\alpha=0=d\beta \) imply that \( \alpha \) and
\( \beta \) are harmonic forms.
 
\begin{remark}
It is curious that the expression \eqref{eq:quat4form} of \( \Omega \) resembles that of a \( 4 \)-form defining a \( \rm{Spin}(7) \)-structure: 
\begin{equation*}
\begin{split}
\Phi&=\tfrac12(\sigma_1^2+\sigma_2^2-\sigma_3^2)\\
&=((-e^{14}+e^{23}-e^{56})\wedge e^7+(e^{125}+e^{136}+e^{246}-e^{345}))\wedge e^8\\
&\qquad+\tfrac12(-e^{14}+e^{23}-e^{56})^2+(-e^{126}+e^{135}+e^{245}+e^{346})\wedge e^7\\
&=\tfrac12(-e^{14}+e^{23}-e^{56})^2+(-e^{126}+e^{135}+e^{245}+e^{346})\wedge e^7\\
&\qquad+(-e^{14}+e^{23}-e^{56})\wedge e^{78}+(e^{125}+e^{136}+e^{246}-e^{345})\wedge e^8.
\end{split}
\end{equation*}
In the above, the second expression of \( \Phi \) emphasizes the inclusion \( \G_2\subset\rm{Spin}(7) \) (fixing a normal direction \( e_8 \)) and the last one is related to \( \SU(3)\subset \G_2 \) (fixing additionally the direction \( e_7 \)); note that each of these rewritings also involves stable forms. In contrast with \( \rm{Sp}(2)\rm{Sp}(1) \), however, an attempt to use the dimensional reduction to seven or six dimensions as a means of solving \( d\Phi=0 \) (by imposing closedness of the involved stable forms) is extremely restrictive and leads to holonomy reductions \( \G_2\subset\rm{Spin}(7) \) and \( \SU(3)\subset\rm{Spin}(7) \), respectively.
\end{remark}

\bigskip

There are obvious obstructions to the existence of a harmonic \( \SO(4) \)-structure on a  \( 7 \)-manifold:

\begin{proposition}
If \( N \) is a \( 7 \)-manifold with an \( \SO(4) \)-structure, then \( N \) is spin. If, in addition, \( N \) is compact and the structure is harmonic then \( H^3 (N)\) and \( H^4(N) \) are non-trivial. 
\end{proposition}
\begin{proof}
The first assertion follows  since \( \SO(4) \) is contained in \( \G_2 \). The statement concerning harmonic structures follows by observing that \( \alpha\w\beta \) is a volume form; in particular, neither \( \alpha \) nor \( \beta \) can be exact. 
\end{proof}

\begin{proposition}
For a  harmonic \( \SO(4) \)-structure, one can identify \( d\upsilon \) with the component \( \Wmod^{1,3} \) of the intrinsic torsion, and  \( d\Hodge\upsilon \) with the components \( \Wmod^{2,2}\oplus\Wmod^{0,4} \).
\end{proposition}
\begin{proof}
At each point,  \( d\upsilon \) is determined by the intrinsic torsion via an \( \SO(4) \)-equivariant map. By Schur's lemma, this implies that \( d\upsilon \) takes values in the component isomorphic to \( \Smod^{1,3} \). Computing the rank of
\begin{equation*} 
  T^*\otimes \so(7)\to\Lambda^4T^*\oplus \Lambda^5T^*\oplus \Lambda^5T^*, \, e^i\otimes A\mapsto e^i\w A\cdot\alpha + e^i\w A\cdot\beta+ e^i\w A\cdot \upsilon 
\end{equation*} 
we conclude that the kernel of \eqref{eqn:intto} does not contain this \( \Smod^{1,3} \).

A similar argument applies  for \( \Hodge\upsilon \). 
\end{proof}

Following \cite{Bryant:CalibratedEmbeddingsIn}, we shall say a subgroup \( \G \) of \( \GL(n,\bR) \) is \emph{admissible} if \( \G \) is the largest subgroup that fixes the space \( (\Lambda^*\bR^n)^{\G} \) of invariant forms on \( \bR^n \). \( \G \) is \emph{strongly admissible} if, on \( \G \)-structures, the closedness of the invariant forms is equivalent to the vanishing of the intrinsic torsion. For instance, \( \Sp(2)\Sp(1)\subset\SO(8) \) is admissible, but not strongly admissible. Similarly, the above observations imply:

\begin{corollary}
 \( \SO(4)\subset\SO(7) \) is admissible, but not strongly admissible.

\end{corollary}

As a final general remark on \( \SO(4) \)-structures, note that for any given harmonic structure, the  product \( \Sp(2)\Sp(1) \)-structure can be Einstein  only when it is Ricci-flat, due to flatness of the \( S^1 \) factor. Thus, the holonomy can reduce to \( \Sp(2)\Sp(1) \) only when the \( 8 \)-manifold is hyper-K\"ahler, and so the \( 7 \)-manifold has holonomy contained in \( \SO(4) \). On the other hand, if the holonomy reduces to \( \SO(4) \), the product \( 8 \)-manifold is hyper-K\"ahler, and so the Ricci tensor is zero. 

In the compact case, we can use standard arguments from Riemannian geometry to completely characterise the \( 7 \)-manifolds whose holonomy reduces to \( \SO(4) \):

\begin{theorem}
\label{thm:intrzero}
Let \( N \) be a compact \( 7 \)-manifold with a torsion-free \( \SO(4) \)-structure. Then, up to a finite cover, \( N \) is isometric to \( \bT^3\times K \), where \( \bT^3 \) is a flat
torus, and \( K \) is either a flat torus or a K3 surface. 
\end{theorem}
\begin{proof}
By a splitting theorem due to Cheeger and Gromoll (see \cite[Corollary 6.67]{Besse:EinsteinManifolds}), the universal cover of \( N \) has the form \( \tilde N\times \bT^k \), where \( \tilde N \) is a compact, simply-connected Ricci-flat manifold, and \( \bT^k \) is a flat torus. If \( N \) is flat then \( k=7 \), so \( N \) is covered by a torus. Otherwise, since \( \tilde N\times \bT^k \to N \) is a local isometry, \( \tilde N\times \bT^k \) has non-trivial holonomy contained in \( \SO(4) \). As the holonomy representation is reducible, the holonomy of \( \tilde N\times \bT^k \), hence of \( \tilde N \), is in fact contained in \( \SU(2) \). In addition, \( \tilde N \) is irreducible, otherwise de Rham's theorem would give \( \tilde N\cong \tilde N'\times\bR \). Also note that \( \tilde N \) cannot be locally symmetric, by Ricci-flatness. According to Berger's holonomy classification, this implies that \( \tilde N \) is a \( 4 \)-dimensional manifold with holonomy \( \SU(2) \), accordingly a K3 surface (cf. \cite[Theorem 7.3.13]{Joyce:CompactManifoldsWith}). 
\end{proof}

\section{Explicit harmonic \( \SO(4) \)-structures}
\label{sec:harmSO4}

We have already noticed that any \( \SO(4) \)-structure determines a \( \G_2 \)-structure with the same underlying metric. Hence, it might be tempting to look for harmonic \( \SO(4) \)-structures on the total space of the bundle of anti-self-dual forms over a self-dual \( 4 \)-manifold \cite{Salamon:RiemannianGeometryAnd}. By duplicating and modifying the construction of \( \G_2 \)-metrics in \cite[Theorem 11.10]{Salamon:RiemannianGeometryAnd}, we can find an \( \SO(4) \)-structure on a domain of the total space of the bundle \( \Lambda^2_-T^*X\stackrel{\pi}\longrightarrow X \) over a self-dual \( 4 \)-manifold \( X \). In contrast with the \( \G_2 \) case, however, the closedness of the defining forms lead to two incompatible equations when \( X \) has non-zero scalar curvature. More precisely, we can always achieve either \( d\alpha=0 \) or \( d\beta=0 \), separately, but not \( d\alpha=0=d\beta \). In addition, the examples always have \( d\upsilon=0 \); one finds that \( d\Hodge\upsilon \) 
vanishes if and only if \( X \) has zero scalar curvature. The fact that we can always achieve \( d\beta=0 =d\upsilon \) is clearly equivalent to the following result by \cite{Salamon:Self-dualityand}:

\begin{proposition}
If \( X \) is a self-dual \( 4 \)-manifold, an open subset of \( \Lambda_-^2T^*X \) admits a cocalibrated
\( \G_2 \)-structure.

\end{proposition}

\subsection{Proof of Theorem \ref{thm:main}}

Another possible setting for producing examples is that of Lie algebras. In fact, we have found two ways of manufacturing harmonic \( \SO(4) \)-structures on semidirect products of Lie algebras. In the following, we explain these constructions and use them to exhibit a number of harmonic structures on nilpotent and solvable algebras; the nilpotent ones are those appearing in Table \ref{table:harmonicliealg}.  

The construction methods are based on relations between harmonic \( \SO(4) \)-structures and certain types of 
half-flat structures on hypersurfaces thereof. Recall that on  \( \bR^6 \) we can define an \( \SU(2,1) \)- or \( \SL(3,\bR) \)-structure by the stabiliser of the forms
\begin{equation*}
  \omega=e^{12}+e^{34}+e^{56}, \, \psi^+=\re (e^1+ie^2)\w(e^3+ie^4)\w(e^5-ie^6)
\end{equation*}
or
\begin{equation*}
  \omega=e^{12}+e^{34}+e^{56}, \, \rho =e^{135}+e^{246},
\end{equation*}
respectively. In each case another  \( 3 \)-form is induced by
\begin{equation*}
  \psi^-=\im  (e^1+ie^4)\w(e^2+ie^5)\w(e^3-ie^6) \textrm{ and } \hat\rho=e^{135}-e^{246}. 
\end{equation*}
The adjective ``half-flat'' refers to the condition that defining forms \( \psi^+ \) (resp. \( \rho \)) and \( \omega^2 \) are closed. 

Notice that if we take an \( \SU(2,1) \)-structure and fix an extra \( 2 \)-form, say  \( \gamma=e^{56} \), the structure group is reduced to \( S(\Un(2)\times\Un(1)) \). 
Also observe that the structure group \( \SL(3,\bR) \) fixes an additional \( 3 \)-form, namely
\begin{equation*}
  \gamma=(e^1-e^2)\w(e^3-e^4)\w(e^5-e^6).
\end{equation*}

Our first construction of harmonic structures is as follows.

\begin{proposition}
\label{prop:nilconA}
Let \( (\alpha,\beta) \) be a harmonic \( \SO(4) \)-structure on a Lie algebra \( \g \). Let \( \eta \) be a unit \( 1 \)-form in \( \Lambda^{0,1} \) such that \( \n=\ker\eta \) is a Lie subalgebra (i.e. \( d\eta\w\eta=0 \)). Then \( \n \) has an induced half-flat \( S(\Un(2)\times \Un(1)) \)-structure such that
\begin{equation*}
  \alpha|_{\n}=\sqrt3\psi^+, \, \beta|_{\n}=\tfrac32\omega^2, \, \gamma=\tfrac18(\eta^\sharp\hook\alpha-\omega).
\end{equation*}
Assume, in  addition, that \( \n \) is an ideal in \( \g \) (i.e \( \eta \) is closed). Then the derivation \( b\in\Der(\n) \) associated to the adjoint action of \( \sqrt3 \eta^\sharp \) on \( \n \) satisfies
\begin{equation}
  \label{eqn:nilconA}  
  b\cdot \psi^+   - d\omega -8d\gamma = 0 =  \tfrac12 b\cdot \omega^2+ d\psi^-.
\end{equation}

Conversely, given a \( 6 \)-dimensional Lie algebra \( \n \) with a derivation \( b \) and  a half-flat \( S(\Un(2)\times\Un(1)) \)-structure satisfying \eqref{eqn:nilconA}. Then the semidirect product \( \g=\n_{\,b}\!\!\rtimes \bR \)  has a harmonic \( \SO(4) \)-structure.
\end{proposition}
\begin{proof}
Choose an adapted frame \( e^1,\dotsc, e^7 \) such that \( \eta=e^7 \). Then
\begin{equation*}
  \alpha|_{\n} = \omega_1\w e^5+\omega_2\w e^6, \, \beta|_{\n} = -\omega_3\w e^{56}-3e^{1234},
\end{equation*}
and consequently
\begin{equation*}
\begin{gathered}
  \psi^+=\re  (e^1+ie^4)\w(e^2-ie^3)\w(\tfrac1{\sqrt3} e^5+\tfrac{i}{\sqrt3} e^6),\\
  \omega =e^{14}-e^{23}-\tfrac13e^{56}, \, \gamma =-\tfrac13e^{56}.
\end{gathered}
\end{equation*}

These forms determine an \( S(\Un(2)\times \Un(1)) \)-structure on \( \n \) which has an adapted coframe
\( (E^1,\dotsc, E^6)=(e^1, e^4, e^2,-e^3, \tfrac1{\sqrt3} e^5,-\tfrac1{\sqrt3} e^6) \),
and is half-flat because \( \psi^+ \) and \( \omega^2 \) are restrictions of closed forms.

If \( \n \) is an ideal in \( \g \), we can express the exterior derivative \( d \) on the exterior algebra over the semidirect product \( \g=\n\rtimes\bR \) as
\begin{equation*}
  d\chi=\eta\w \ad^*(\eta^\sharp)\cdot \chi + d_{\n}\chi=\tilde\eta\w b\cdot \chi + d_{\n}\chi,
\end{equation*}
where \( \tilde\eta=\tfrac1{\sqrt3}\eta \).

In terms of the adapted coframe, we have
\begin{equation*}
  \psi^-=\im (e^1+ie^4)\w(e^2-ie^3)\w(\tfrac1{\sqrt3} e^5+\tfrac{i}{\sqrt3} e^6)=\tfrac1{\sqrt3}(e^{126}-e^{135}+e^{425}+e^{436}), 
\end{equation*}
so that \( \alpha \) and \( \beta \) can be expressed as
\begin{equation*}
  \alpha = \sqrt3\left(\psi^+ + \tilde\eta\w(\omega+8\gamma)\right),\, \beta =  3\left(\tfrac12\omega^2  -\psi^-\w \tilde\eta\right).
\end{equation*}
Thus
\begin{equation*}
  \eta^\sharp\hook d\alpha =  b\cdot\psi^+ - d_{\n}\omega -8 d_{\n}\gamma,\, \eta^\sharp\hook d\beta = \sqrt3\left( \tfrac{1}2 b\cdot \omega^2- d_{\n}\psi^-\right),
\end{equation*}
from which \eqref{eqn:nilconA} follows.

For the last part of the statement, let \( \eta \) be a \( 1 \)-form on \( \g=\n\rtimes\bR \) that annihilates \( \n \) and satisfies
\( d\chi=\tfrac1{\sqrt3}\eta\w b\cdot \chi + d_{\n}\chi \).

If \( E^1,\dotsc, E^6 \) is a coframe on \( \n \), adapted to the given  \( S(\Un(2)\times \Un(1)) \)-structure, set
\begin{equation*}
  (e^1,\dotsc, e^7)=(E^1, E^3,-E^4, E^2, \sqrt3 E^5, -\sqrt3E^6, \eta).
\end{equation*}
The considerations above imply that the \( \SO(4)\)-structure on \( \g \) associated to this coframe is harmonic. 
\end{proof}

\begin{remark}
The \( 2 \)-form \( \gamma \) can be expressed in terms of \( \Hodge \upsilon \). Inspection shows that \( \gamma=\eta^\sharp\hook\Hodge\upsilon \).
In particular, it follows that if \( d\Hodge\upsilon=0 \), then \( \gamma \) is closed. In that case, \eqref{eqn:nilconA} reads
\begin{equation*}
  b\cdot \psi^+ = d\omega, \, b\cdot \omega^2=-2 d\psi^-.
\end{equation*}
These equations are algebraic analogues of the flow equations one encounters in the study of half-flat \( \SU(3) \)-structures
in the context of \( \G_2 \)-holonomy metrics (cf. \cite{Hitchin:Stableformsand}).
\end{remark}

The method described above enables us to construct harmonic \( \SO(4) \)-structures on two nilpotent 
Lie algebras, corresponding to the first two algebras of Table \ref{table:harmonicliealg}.

\subsubsection{The algebra \( \n_1 \)}
The nilpotent Lie algebra 
\begin{equation*}
(0, 0, f^{27}, f^{27}{+}f^{17}, f^{12}, f^{14}{+}f^{23}{+}f^{57}, 0)
\end{equation*}
admits a harmonic \( \SO(4) \)-structure whose adapted coframe
is given by
\begin{equation*} 
(e^1,\ldots,e^7)=(f^1+f^2, f^6, - f^5, f^3, \sqrt3 f^4, -\sqrt3 f^1, \sqrt3 f^7). 
\end{equation*}

In terms of the classifications of \cite{Gong:Classificationofnilpotent,Conti-F:Nilmanifoldswithcalibrated}, this Lie algebra is the distinguished member of the \( 1 \)-parameter family of Lie algebras, denoted by \( 147E1 \), which carries a calibrated \( \G_2 \)-structure. Specifically, it can be written as
\begin{equation*} 
\n_1=(0,0,0,12,23,-13,26-34-16+25). 
\end{equation*}

\subsubsection{The algebra \( \n_2 \)}
\label{example:giovannini}
The second example of a harmonic structure is found on the Lie algebra \(  (0, 0, 0, 0, f^{27}{+}f^{12}, f^{37}{+}f^{13}, 0) \).
The adapted coframe is
\begin{equation*} 
  (e^1,\ldots,e^7)=(f^2-f^6, f^1, -f^4, f^5, \sqrt3f^3 , - \sqrt3f^6, \sqrt3f^7).
\end{equation*}

The underlying Lie algebra can also be expressed as 
\( \n_2=(0,0,0,0,0,12,13) \)
and is clearly a decomposable Lie algebra. In fact, it can be decomposed as 
\begin{equation*}
  \Span{f^1,f^2,f^3,f^5,f^6,f^7}\oplus\Span{f^4};
\end{equation*}
the \( 6 \)-dimensional component inherits an \( \SO(3) \)-structure, in the sense of \cite{Giovannini:SpecialStructuresAnd}. Its  defining forms are
\begin{equation*} 
  -e^{45}-e^{16}+e^{27}, \, e^{125}-e^{426}+e^{147}-3e^{567},\, -e^{467}-e^{175}-e^{256}+3e^{124}.
\end{equation*}
The associated \( \Sp(2)\Sp(1) \)-structure is already known due to  \cite{Salamon:Almostparallelstructures,Giovannini:SpecialStructuresAnd}.

In this case, there is a \( 4 \)-dimensional subalgebra \( \so(3)\oplus\lie{u}(1) \) of \( \so(8) \) that preserves the Lie algebra structure; its intersection with \( \spto \) is the \( 1 \)-dimensional
component. By construction, the infinitesimal action of \( \so(3) \) on the closed \( 4 \)-form gives a space of closed forms; the infinitesimal orbit under the action of \( \so(8) \) contains no other closed form.

\bigskip

Similar techniques as in the proof of the Proposition \ref{prop:nilconA} lead to:

\begin{proposition}
\label{prop:nilconB}
Let \( (\alpha,\beta) \) be a harmonic \( \SO(4) \)-structure on a Lie algebra \( \g \). Let \( \eta\in\Lambda^{1,0} \) be a unit \( 1 \)-form on \( \g \) such that \( \n=\ker\eta \) is a subalgebra. Then \( \n \) has an induced half-flat \( \SL(3,\bR) \)-structure such that
\begin{equation*} 
  \alpha|_{\n}=\tfrac{\sqrt6}3\rho, \, \beta|_{\n}=-\tfrac16\omega^2. 
\end{equation*}
Assume, in addition, that \( \n \) is an ideal. Then the derivation \( b\in\Der(\n) \) associated to \( \ad(\sqrt 2 \eta^\sharp) \) satisfies
\begin{equation}
  \label{eqn:nilconB}
  b\cdot \rho   + d\omega =0,  \, -\tfrac{1}{4} b\cdot \omega^2+ d\hat\rho+	2 d\gamma=0.
\end{equation}

Conversely, given a \( 6 \)-dimensional Lie algebra \( \n \) with a derivation \( b \) and  a half-flat \( \SL(3,\bR) \)-structure satisfying \eqref{eqn:nilconB}, then the semidirect product \( \g=\n_{\,b}\!\!\rtimes \bR\) has a harmonic \( \SO(4) \)-structure.

\end{proposition}

\begin{remark}
Expressing the \( 3 \)-form \( \gamma \) in terms of \( \upsilon \), 
\( \gamma=2\sqrt2\eta^\sharp\hook\upsilon \), we see that if \( d\upsilon=0 \) 
then \( \gamma \) is closed. Consequently, the system \eqref{eqn:nilconB} reduces to the equations
\begin{equation*}
    b\cdot \rho =-d\omega ,\,  \tfrac14 b\cdot \omega^2= d\hat\rho.
\end{equation*}
\end{remark}

The method of Proposition~\ref{prop:nilconB} results in harmonic \( \SO(4) \)-structures on \( 9 \) different nilpotent 
Lie algebras; these are the last \( 9 \) algebras of Table \ref{table:harmonicliealg}. 
In addition, we find harmonic structures on different (non-nilpotent) solvable Lie algebras, as explained in the next
section.

\subsubsection{The algebras \( \n_3,\ldots,\n_7\)}
 A harmonic structure is found on the family of Lie algebras given by
\begin{equation*}
  \begin{split}
  \bigl(0,0, (cf^1- bf^2)f^7&, (3  bf^2-2  cf^1)f^7+f^{12},(af^1 +3f^2{-}3  b f^3- bf^4)f^7+f^{13},\\
  &-(3f^1+af^2+cf^4+2  c f^3)f^7+f^{23}, 0\bigr) 
  \end{split}
\end{equation*}
where \( a,b,c\in \bR \). The adapted coframe is
\begin{equation*}
  \tfrac1{\sqrt2}(f^2-f^5, f^6-f^1, -f^3,  2f^7, -\tfrac1{\sqrt3} (5 f^3+2 f^4), \tfrac1{\sqrt3}(f^1+f^6), \tfrac1{\sqrt3}(f^2+f^5)).
\end{equation*}

There are three cases, depending on the value of \( b \) and \( c \). First, let us assume both \( b \) and \( c \) are zero. Then the space of exact \( 2 \)-forms is a \( 3 \)-dimensional subspace  \( V\subset \Lambda^2\Span{f^1,f^2,f^3,f^7} \). The form \( df^4 \) is decomposable, and the subset \( \{\gamma\in V\mid \gamma^2=0\} \) is the union of straight lines parallel to \( df^4 \). Taking its image in the quotient space \( V/\Span{df^4} \), we obtain the trivial space when \( \abs{a}<3 \), a line when \( \abs{a}=3 \), and two lines when \( \abs{a}>3 \). Now, using the classification of nilpotent Lie algebras, we conclude that our Lie algebra is isomorphic, respectively, to 
\begin{equation*}
  \begin{gathered}
  \n_3=(0,0,0,0,12-34,13+24,14), \,\n_4=  (0,0,0,0,12+34,23,24),\\ 
  \n_5= (0,0,0,0,12,13,34).
  \end{gathered}
\end{equation*}

Secondly, let us consider the case when exactly one of \( b \) or \( c \) is zero. Then the Lie algebra is isomorphic to \( \n_6=(0,0,0,12,13,23+14,25+34) \). This can be seen, e.g., by considering the coframe 
\begin{equation*}
  (f^2,bf^7,-\tfrac3bf^1,-f^3,\tfrac3b(f^4+3f^3-\tfrac ab f^1),-f^6+\tfrac ab f^3,\tfrac3b(f^5+\tfrac3bf^3)) 
\end{equation*}
when \( b\neq0 \) and 
\begin{equation*}
  (f^1,-cf^7,-\tfrac 3cf^2,-f^3,-\tfrac3c(f^4+2f^3+\tfrac ac f^2),-f^5 +\tfrac acf^3,\tfrac 3c(f^6+\tfrac 3c f^3) )
\end{equation*}
when \( c\neq0 \).

Finally, we are left with the case when  \( b \), \( c \) are both non-zero. Then the first Betti number is three and, in addition, the dimension of
\( [\g,[\g,\g]] \) is two. Going through the classification of \cite{Gong:Classificationofnilpotent} (see also \cite{Conti-F:Nilmanifoldswithcalibrated}), this leaves us with
\( 22 \) possibilities. We can easily rule out \( 14 \) of these, namely by observing that the space
\( \Span {df^i\w df^j\mid 1\leq i,j\leq 7} \) has dimension two. For each of the remaining \( 8 \) possibilities, we then consider the unique (up to non-zero multiple) element \( \tilde f \) of \( \g^* \) satisfying 
\begin{equation*}
  \dim \left(B^2\w\tilde f + \Lambda^3([\g,\g]^o)\right)=2.
\end{equation*}
In our case, \( \tilde f=-cf^1+bf^2-bcf^7 \) and \( \tilde f\w d([[\g,\g],\g]^o)=\Span{bcf^{127}} \). Since, by assumption, we have \( bc\neq0 \), this space is \( 1 \)-dimensional. These observations leave us with only two possibilities which are
\begin{equation*}
  (0,0,0,12,13,35+24+15+14,34+25),\, \n_7=(0,0,0,12,13,35+15,34+25). 
\end{equation*}
Considering finally the dimension of \( B^3(\g) \), we can rule out the first of these.

\subsubsection{The algebras \( \n_8,\ldots,\n_{10} \)}  
\label{example:F1}
The family of nilpotent Lie algebras
\begin{equation*} 
  (0, 0, 0, 0, af^{27}+f^{12}+2 f^{37}, -af^{37}+e^{13}-2 f^{27}, 0), 
\end{equation*}
\( a\in\bR \), admits a harmonic \( \SO(4) \)-structure with adapted coframe given by
\begin{equation*} 
  \tfrac1{\sqrt2}(f^3-f^5,f^6-f^2,f^1,2f^7,-\tfrac1{\sqrt3}(f^1+2f^4),\tfrac1{\sqrt3}(f^2+f^6),\tfrac1{\sqrt3}(f^3+f^5)). 
\end{equation*}
This is a decomposable Lie algebra which is isomorphic to 
\begin{equation*} 
  \begin{gathered}
    \n_8=(0,0,0,0,0,12,34),\,\mathfrak n_9= (0,0,0,0,0,12,14+23),\\ 
    \n_{10}= (0,0,0,0,0,13+42,14+23),  
 \end{gathered}
\end{equation*}
according to whether \( \abs{a}>2 \), \( a=\pm2 \), or \( \abs{a}<2 \). We might as well assume that \( a\geq 0 \), since the symmetry
\( f^1\mapsto -f^1 \),\( f^2\mapsto -f^2 \),\( f^6\mapsto -f^6 \), \( f^4\mapsto -f^4 \) corresponds to an element of \( \Sp(2)\Sp(1) \) 
whose only effect is to change the sign of \( a \).

As in the section~\ref{example:giovannini}, we can choose a decomposition of the Lie algebra so as to obtain an \( \SO(3) \)-structure on the \( 6 \)-dimensional component. The defining forms, however,
will not be closed in this case. This explains why the algebra is absent from \cite{Giovannini:SpecialStructuresAnd}.

A straightforward computation shows that, in terms of the adapted frame, the Ricci tensor of the associated \( 8 \)-dimensional metric is given by
\begin{equation*}
\left(\begin{smallmatrix}
    \tfrac{4}3+\tfrac{a^2}6              &\tfrac{4a}3                            &0          &0                        &0       &0                                   &-\tfrac{4\sqrt3}3-\tfrac{\sqrt3a^2}6  &0\\
    \tfrac{4a}3                          &\tfrac43+\tfrac{a^2}6                  &0          &0                        &0       &\tfrac{4\sqrt3}3+\tfrac{\sqrt3a^2}6 &0                                     &0\\
    0                                    &0                                      &-\tfrac83  &0                        &0       &0                                   &0                                     &0\\
    0                                    &0                                      &0          &-\tfrac83-\tfrac{2a^2}3  &0       &0                                   &0                                     &0\\
    0                                    &0                                      &0          &0                        &0       &0                                   &0                                    &0\\
    0                                    &\tfrac{4\sqrt3}3 +\tfrac{\sqrt3a^2}{6} &0          &0                        &0       &-\tfrac43-\tfrac{a^2}6               &-\tfrac{4a}3                         &0\\
    -\tfrac{4\sqrt3}3-\tfrac{\sqrt3a^2}6 &0                                      &0          &0                        &0       &-\tfrac{4a}3                         &-\tfrac43-\tfrac{a^2}6               &0\\
    0                                    &0                                      &0          &0                        &0       &0                                    &0                                    &0
\end{smallmatrix}\right);
\end{equation*}
we give the Ricci tensor explicitly in this case, because it will be of importance for us as a means of comparing the algebras appearing in this section with those found in the next.

In accordance with \cite{Witt:Specialmetricsand}, the component of the traceless Ricci tensor that commutes with the almost-complex structures of the quaternionic structure is zero.

By studying the infinitesimal action of \( \so(8) \), we find that all the closed \( 4 \)-forms are induced by automorphisms:
for \( a=0 \), the subalgebra that preserves the Lie bracket is \( \bR^2\oplus\so(3) \); its intersection with \( \spto \) is \( 1 \)-dimensional, giving rise to a \( 3 \)-dimensional space of closed \( 4 \)-forms. If \( a=\pm4 \), the subalgebra that preserves the Lie bracket is \( \bR^2 \); it intersects \( \spto \) trivially. It gives rise to a \( 2 \)-dimensional space of closed \( 4 \)-forms.
Finally, in the case when \( a\neq0,\pm4 \), there is a \( \rm{U}(1) \) of automorphisms that acts as a circle action on \( \Span{f_4,f_8 }\); this \( \rm{U}(1) \) clearly does not preserve the \( 4 \)-form, so we find a \( 1 \)-dimensional space of closed \( 4 \)-forms. 

An explicit calculation and a dimension count show that in each case, the infinitesimal orbit contains no other closed \( 4 \)-forms. 

\subsubsection{The algebra \( \n_{11} \)} 
\label{ex:harmfinal}
The proof of Theorem~\ref{thm:main} is completed by considering the family of nilpotent Lie algebras 
\begin{equation*} 
  (-af^{27}, 0, af^{27}, 0, - b f^{47}-a f^{17} -a f^{37}-f^{27}, b f^{27}+f^{47}+f^{12}, 0), 
\end{equation*}
carrying an harmonic \( \SO(4) \)-structure with adapted coframe
\begin{equation*} 
  \tfrac1{\sqrt2}(f^4-f^6, f^5-f^2, f^1, 2f^7, -\tfrac1{\sqrt3}(f^1+2 f^3), \tfrac1{\sqrt3}(f^2+f^5), \tfrac{1}{\sqrt3}(f^4+f^6)).
\end{equation*}

For the parameter values \( a=0=b \), the Lie algebra is  decomposable and isomorphic to 
\( (0,0,0,0,0,12,14+23) \). When \( a=0\neq b \), we obtain \( (0,0,0,0,0,12,34) \).
Finally, if \( a\neq 0 \), we see that the underlying Lie algebra is isomorphic to
\begin{equation*} 
  \n_{11}=(0,0,0,0,12,13,14+25). 
\end{equation*}
 Again, in the reducible case, the structures cannot be decomposed so as to fit into the construction of \cite{Giovannini:SpecialStructuresAnd}.

If we set \( a=0 \), the Lie algebra is
\begin{equation*} 
  (0, 0, 0, 0, - b f^{47}-f^{27}, b f^{27}+f^{47}+f^{12}, 0).
\end{equation*}
We can assume that \( b\geq 0 \), since the symmetry \( f^1\mapsto -f^1 \), \( f^2\mapsto -f^2 \), \( f^3\mapsto -f^3 \), \( f^5\mapsto -f^5 \)
is an element of \( \SO(4) \) with the only effect that it changes the sign of \( b \).

Note that the resulting \( \SO(4) \)-structures are not the same as those that appear in section~\ref{example:F1}. However,  the resulting harmonic \( \Sp(2)\Sp(1) \)-structures are related through an automorphism of the \( 8 \)-dimensional Lie algebra which only acts on the \( 2 \)-dimensional center.

More precisely, denote by \( e_1,\dotsc, e_7 \) the adapted frame we have fixed on the present Lie algebra, and by \( \tilde e_1,\dotsc, \tilde e_7 \) the adapted frame considered in section~\ref{example:F1}. If we now set \( a=2\sqrt{b^2+1} \), we can construct a linear isomorphism between the two Lie algebras that maps eigenvectors of the Ricci tensor to eigenvectors of the same eigenvalues, as follows:
\begin{equation*}
  \begin{gathered}
  \tilde e_3\mapsto \tfrac2 a(be_3{+}e_4), \, \tilde e_4\mapsto \tfrac2a({-}e_3{+}be_4), \, \tilde e_5\mapsto e_5,\\
  \tilde e_1{\mp}\tilde e_2 {\pm}\sqrt3 \tilde e_6{+}\sqrt3\tilde e_7 \mapsto  \tfrac1{\sqrt{a}}\left(\sqrt{a{\mp}2}e_1 {\mp}\sqrt{a{\pm}2} e_2 {\pm} \sqrt3\sqrt{a\pm2} e_6{+}\sqrt3 \sqrt{a\mp2}e_7\right),\\
  {\mp}\sqrt3\tilde e_1{+}\sqrt3\tilde e_2 {+} \tilde e_6{\pm}\tilde e_7 \mapsto  \tfrac1{\sqrt{a}}\left({-}\sqrt3 \sqrt{a{\mp}2} e_1 {\pm}\sqrt3\sqrt{a{\pm}2} e_2{\pm} \sqrt{a{\pm}2}e_6{+}\sqrt{a{\mp}2}e_7\right).
   \end{gathered}
\end{equation*}
Up to a scale factor, this mapping is a Lie algebra isomorphism, and it maps the \( 4 \)-form \( \tilde\Omega \), corresponding to section~\ref{example:F1}, to
\begin{equation*}
  \begin{gathered}
    e^{3467}- e^{3458}-3  e^{5678}- e^{1267}-3  e^{1234}+e^{1258}+e^{147}(ke^8- 2he^5)\\
    +e^{23}\w (k e^5+2 he^8 ) \w e^6
    +e^{246}(ke^8- 2he^5)+e^{24}(ke^5+2he^8)e^{7}\\
    +e^{136}(ke^8-2he^5)
    - e^{14}(ke^5+ 2he^8)e^6+e^{13}(ke^5+ 2he^8)e^{7}- e^{237}(ke^8- 2he^5),
  \end{gathered}
\end{equation*}
where \( k  =\tfrac{a^2-8}{a^2} \) and \( h=2\tfrac{\sqrt{a^2-4}}{a^2} \). As \( k^2+4h^2=1 \), it now suffices to set 
\begin{equation*}
  \tilde e_5=k e_5 -2he_8, \,\tilde e_8 =  2h e_5+ke_8,
\end{equation*}
so as to obtain an isomorphism, up to scale, of the two \( 8 \)-dimensional Lie algebras that preserves the  \( \Sp(2)\Sp(1) \)-structures. However, the \( \SO(4) \)-structures are not equivalent because this transformation mangles the possibility to separate \( e_5 \) from \( e_8 \).

\bigskip

This completes the proof of Theorem \ref{thm:main}.

\qed

\subsection{Examples on solvable Lie algebras}

Using the same procedures as above enables us to find harmonic \( \SO(4) \)-structures on the following families of solvable Lie algebras.

\begin{example}
\label{ex:firstsolv}
The family of completely solvable unimodular Lie  algebras
\begin{equation*} 
  (c f^{27}{-}4 f^{17}, 4  f^{27},{-}4  f^{37}, 4 f^{47}{-} a f^{37},  b f^{27}{+} a f^{17}{+}f^{12}{+}f^{37}, f^{34}{-} b e^{37}{-}c e^{47}{-}f^{27}, 0),
\end{equation*}
\( a,b,c\in\bR \), admits a harmonic structure whose adapted coframe is given by
\begin{equation*}
  \tfrac1{\sqrt2}(f^5-f^3, f^2-f^6, f^1-f^4,2f^7, -\tfrac1{\sqrt3}{(f^1+f^4)}, \tfrac1{\sqrt3}{(f^2+f^6)}, \tfrac1{\sqrt3}{(f^3+f^5)}). 
\end{equation*}
This determines a harmonic structure on a compact solvmanifold if the corresponding Lie group admits a lattice.

In order to address the existence of a lattice, we first observe that each Lie group in this family is an almost nilpotent Lie group of the form \( \G=\bR\ltimes_\mu \rm{N} \) where \( \rm{N} \) is nilpotent and
\begin{equation*}
\mu(t)=\exp^{\rm{N}}\circ (\exp tA)\circ\log^{\rm{N}}, \quad A=\ad f_7\colon \n\to\n.
\end{equation*}

For simplicity, we will only consider the case when \( a=b=c=0 \). Prompted by \cite[Equation (4)]{Bock:OnLowDimensional}, we set \( e^{4t_1}=\frac12(3+\sqrt5) \) and consider the basis
\begin{equation*}
\begin{gathered}
\tilde f_1=f_1-4f_2+f_6,\quad \tilde f_2=\tfrac{18+8 \sqrt{5}}{7+3 \sqrt{5}}f_1 -\tfrac{8}{3+\sqrt{5}}f_2+ \tfrac{2}{3+\sqrt{5}}f_6,\\
\tilde f_3=-4f_3+f_4+f_5,\quad  \tilde f_4=\tfrac{-4(18+8 \sqrt{5})}{7+3 \sqrt{5}}f_3 +\tfrac{2}{3+\sqrt{5}}f_4+\tfrac{18+8 \sqrt{5}}{7+3 \sqrt{5}}f_5,\\
\tilde f_5=\sqrt 5 f_5, \quad \tilde f_6=\sqrt 5 f_6.
\end{gathered}
\end{equation*}

The automorphism \( \exp (t_1A) \) is then represented by the matrix
\begin{equation}
 \label{eqn:mut}
\exp (t_1A)=\left(
\begin{smallmatrix}
 3 & 1 & 0 & 0 & 0 & 0 \\
 -1 & 0 & 0 & 0 & 0 & 0 \\
 0 & 0 & 3 & 1 & 0 & 0 \\
 0 & 0 & -1 & 0 & 0 & 0 \\
 0 & 0 & 0 & 0 & 1 & 0 \\
 0 & 0 & 0 & 0 & 0 & 1 \\
\end{smallmatrix}
\right). 
\end{equation}
Since the structure constants of \( N \) are given by 
\begin{equation*}
[\tilde f_1,\tilde f_2]=4\tilde f_5, \quad [\tilde f_3,\tilde f_4]=4\tilde f_6,
\end{equation*}
we can make the explicit identification \( \rm{N}\cong\bR^6 \) by defining the product law as
\begin{equation*}
(x_1,\dotsc, x_6)(y_1,\dotsc, y_6)=(x_1+y_1,\ldots,x_4+y_4,x_5+y_5-x_1y_2,x_6+y_6-x_3y_4),
\end{equation*}
and identify the basis \( \{\tilde f_j\} \) with the left-invariant vector fields 
\begin{equation*}
\begin{gathered}
\tilde f_1=\frac{\partial }{\partial {x_1}},\quad\tilde f_2=\frac{\partial }{\partial {x_2}}-x_1\frac{\partial }{\partial {x_5}},\quad\tilde f_3=\frac{\partial }{\partial {x_3}},\\
\tilde f_4=\frac{\partial }{\partial {x_4}}-x_3\frac{\partial }{\partial {x_6}},\quad \tilde f_5=\frac14\frac{\partial }{\partial {x_5}},\quad \tilde f_6=\frac14\frac{\partial }{\partial {x_6}}.
\end{gathered}
\end{equation*}

The exponential map \( \exp^{\rm{N}}(t a_i \tilde f_i) \) is determined by solving the ODE \( g'(t)=L_{g(t)*}(a_i\tilde f_i) \). Concretely, we have 
\begin{equation*}
\exp^{\rm{N}}(a_i \tilde f_i)=(a_1, \ldots, a_4, \tfrac14a_5-\tfrac12a_1a_2, \tfrac14a_6-\tfrac12a_3a_4),
\end{equation*}
and
\begin{equation*}
\log^{\rm{N}}(x_1,\dotsc, x_6)=x_1\tilde f_1+\dots+x_4\tilde f_4+(4x_5+2x_1x_2)\tilde f_5+(4x_6+2x_3x_4)\tilde f_6.
\end{equation*}

Since the lattice \( \Gamma\subset\rm{N} \) corresponding to the inclusion \( \bZ^6\subset\bR^6 \) is clearly preserved by \( \mu(t_1) \), the discrete subgroup 
\begin{equation*}
\bZ t_1\ltimes_\mu \Gamma
\end{equation*}
forms a lattice in \( \G=\bR\ltimes_\mu \rm{N} \).
\end{example}

\begin{example}
A harmonic structure can be put on the family of unimodular solvable Lie algebras given by
\begin{equation*} 
  (2 f^{17}, {-}4 f^{27}, {-} a f^{27}{-}2 f^{37}{-}f^{17}{-} b f^{47}, 6  f^{47},f^{12}{+} bf^{17}{-}2 f^{57},  a f^{17}{+}f^{27}{+}f^{15}, 0), 
\end{equation*}
for \( a,b\in\bR \). The associated coframe is
\begin{equation*}
  \tfrac{1}{\sqrt2}  (f^6{-}f^2, f^5{-}f^4, f^1{-}f^3,  2f^7,{-}\tfrac{1}{\sqrt3}{(f^1{+}f^3)}, \tfrac{1}{\sqrt3}{(f^4{+}f^5)}, \tfrac{1}{\sqrt3}{(f^2{+}f^6)}).
\end{equation*}
It would be interesting to know whether the corresponding solvable Lie groups admit a lattice; this case appears to be more involved than Example~\ref{ex:firstsolv} because the nilradical has step three.
\end{example}

\section{Invariant intrinsic torsion}
\label{sec:invtor}

The intrinsic torsion of an \( \SO(4) \)-structure is dubbed \emph{invariant} if it takes values in the trivial submodule \( 2\bR\subset T^*\otimes\so(4)^\perp \). We may think of the latter space as the cokernel of the alternating map
\begin{equation}
  \label{eqn:alternating}
 \partial\colon\, T^*\otimes \so(4)\hookrightarrow T^*\otimes T^*\otimes T\to \Lambda^2T^*\otimes T.
\end{equation}
Then the trivial module \( 2\bR \) is the image of \( \Span{\tau_1,\tau_2} \subset T^*\otimes\so(7)\), where
\begin{equation*}
  \tau_1 =\sum_{i=1}^7e^i\otimes e_i\hook\alpha, \, \tau_2 = w^1\otimes w^{23}+w^2\otimes w^{31}+w^3\otimes w^{12},
\end{equation*}
and we have identified vectors with covectors using the metric. 

Using the alternation map \( \mathbf{a} \) (cf. section \ref{sec:SO4str}), we compute
\begin{equation*}
  \mathbf{a}(\tau_1\cdot\alpha)=-2\beta+6\nu,\,\mathbf{a}(\tau_2\cdot\alpha)=2\beta+6\nu,
\end{equation*}
\begin{equation*}
  \mathbf{a}(\tau_1\cdot *\upsilon)=2\beta+6\nu,\,\mathbf{a}(\tau_2\cdot *\upsilon)=0.
\end{equation*}
Thus if the intrinsic torsion has the form \( \lambda\tau_1 + \mu\tau_2 \), then we obtain
\begin{equation*}
  d\alpha = 2(-\lambda+\mu)\beta+6(\lambda+\mu) \upsilon, \, d(\Hodge\upsilon) = 2\lambda(\beta + 3\upsilon).
\end{equation*}
In particular, the ``\(  \lambda \) component'' of the intrinsic torsion is completely determined
by \( d(\Hodge\upsilon) \). 

Since there are no invariants of degree \( 5 \), the forms \( \upsilon \) and \( \beta \) are closed.
Using this fact and by imposing \( d^2=0 \), we obtain the equations
\begin{equation*} 
  2(-d\lambda+d\mu)\w \beta+6(d\lambda+d\mu) \w \upsilon = 0 = 2d\lambda\w (\beta + 3\upsilon),
\end{equation*}
which, by injectivity of the map \( T^*\ni\delta\mapsto \delta\wedge(\beta+3\upsilon)\in\Lambda^5T^* \), implies that \( \lambda \) and \( \mu \) are necessarily constant.

\begin{proposition}
\label{prop:consttor}
Any \( \SO(4) \)-structure with invariant intrinsic torsion has constant torsion.

\end{proposition}

\subsection{The proof of Theorem \ref{thm:invintrtor}}

The simplest examples of \( \SO(4) \)-structures with invariant intrinsic torsion are given by Lie groups \( \LH \). Let \( e^1,\dotsc, e^4\), \(w^1,w^2,w^3 \) be a basis of \( \h^* \) adapted to an \( \SO(4) \)-structure. The torsion \( \tau \) of the flat connection can be expressed as \( dw^s\otimes w^s +de^j\otimes e^j\), and its projection to \( T^*\otimes\so(4)^\perp \) is invariant if and only if \( \tau \) lies in \( 2\bR\oplus \partial(T^*\otimes\so(4)) \), where \( \so(4) \) is given by \eqref{eq:so4stab}, and the \( 2\bR \) component is 
\begin{equation*}
 \vartheta = a\omega_s\otimes w^s + c e_j\hook\omega_s\w w^s\otimes e^j.
\end{equation*}
Throughout the section, we shall use the Einstein summation convention.

\subsubsection{The algebra \( \h_1 \)}
Consider the element \( \sum_{s,t} w^s\otimes\varpi_t \) of the module \( \Smod^{2,2}\subset T^*\otimes \so(4) \). 
Its image under the map \( \partial \) is given by:
\begin{equation*}
\begin{gathered}
(e^2+e^3+e^4 )\w w \otimes e^1-(e^1-e^3+e^4)\w w\otimes e^2\\
-(e^1+e^2-e^4)\w w\otimes e^3-(e^1-e^2+e^3)\w w\otimes e^4,
\end{gathered}
\end{equation*}
where \( w=w^1+w^2+w^3 \). It follows that the solvable Lie algebra \( \h_1 \), given by
\begin{equation*}
\begin{gathered}
de^1=(e^2+e^3+e^4)\w w,\, de^2=-(e^1-e^3+e^4)\w w,\\
de^3=-(e^1+e^2-e^4)\w w,\, de^4=-(e^1-e^2+e^3)\w w\\
dw^1=0 =dw^2=dw^3,
\end{gathered}
\end{equation*}
has an \( \SO(4) \)-structure with vanishing intrinsic torsion. 

The three \( 2 \)-forms of \eqref{eq:sp2forms} are globally defined and hence determine an almost hypercomplex structure on \( S^1\times \LH_1 \) which is compatible with the product metric. In other words, we have an \( \Sp(2) \)-structure on \( S^1\times \LH_1 \). As \( d\sigma_i=0 \), we see that this structure is, in fact, hyper-K\"ahler and flat (cf. \cite[Lemma 6.8]{Htichin:TheSelfDuality}).

\bigskip

To obtain examples with a non-flat metric, consider next the element
\begin{equation*}
\begin{gathered}
\chi=w^1\otimes(e^{12}-e^{34}-2w^{23})+w^2\otimes(e^{13}-e^{42}-2w^{31})\\
              +w^3\otimes(e^{14}-e^{23}-2w^{12})
\end{gathered}
\end{equation*}
that spans the trivial component \( \bR\subset \Smod^{0,2}\otimes\Smod^{0,2}\subset T^*\otimes \so(4) \). Straightforward computations show that
\begin{equation*}
\begin{gathered}
\tau=\vartheta+b\partial(\chi)=(a\omega_s-4bw_s\hook w^{123})\otimes w^s
                                +(b+c)(e_j\hook \omega_s\w w^s)\otimes e^j.
\end{gathered}
\end{equation*}

Now, let us write \( p=a \), \(q=-4b \), and \( r=b+c \) and, in addition, impose \( d^2=0 \) so as to get:
\begin{lemma}
If \( p,q,r \in \bR \) satisfy the equations 
\begin{equation}
  \label {eq:invLiedsq}
  pq=0=pr =r(q+2r),
\end{equation}
then the relations
\begin{equation}
 \label{eqn:CITstructureeqns}
 de^j=re_j\hook\omega_s\w w^s,\, dw^s = p\omega_s + qw_s\hook w^{123}
\end{equation}
are the structural equations of a Lie algebra with an \( \SO(4) \)-structure
whose intrinsic torsion is invariant. 
\end{lemma} 

Note that, in terms of \( a,b,c \), we find that \( \lambda=-a\slash2 \) and \( \mu=a-2c \). In particular,
the intrinsic torsion is independent of the component \( \partial(\chi) \) of \( \tau \), as expected. 

Up to a scale factor, \eqref{eqn:CITstructureeqns} gives rise to exactly three non-trivial examples. 

\subsubsection{The algebra \( \h_2 \)}
If we take \( (p,q,r)=(0,1,0) \) then \( de^i=0 \), \( dw^s =w_s\hook w^{123} \) which characterises
the product \( \LH_2=\Sp(1)\times \mathbb{T}^4 \). Using the relations
\begin{equation*}
\label{eq:invtorparam}
\lambda=-\tfrac{p}2,\,\mu=p-\tfrac{q}2 -2r,
\end{equation*}
we see that \( (\lambda,\mu)=(0,-1\slash2) \) and hence \( d\alpha=-\beta-3\upsilon \), \( d(\Hodge\upsilon)=0 \).

The product of \( \LH_2 \) with an \( S^1 \), \( M=S^1\times\Sp(1)\times\mathbb{T}^4 \), can also be viewed as the product of two quaternionic manifolds of dimension \( 4 \); one of the factors is hyper-K\"ahler and the other is locally conformal hyper-K\"ahler. It is not difficult to verify that the \( 2 \)-forms \( \sigma_i \) define a hypercomplex structure on \( M \). In fact, the associated hyper-Hermitian structure is ``hyper-K\"ahler with torsion'' (briefly \HKT), which, by \cite[Proposition 6.2]{Cabrera-S:TheIntrinsicTorsion}, means it satisfies the condition \( I_1d\sigma_1 = I_2d\sigma_2=I_3d\sigma_3 \). 

In the notation of \cite{Swann:AspectsSymplectiquesDe}, the two factors of \( M \) have intrinsic torsion in \( EH \), and the product has intrinsic torsion in \( EH\oplus KH \). This is consistent with the fact that the ideal generated by \( \sigma_1 \), \( \sigma_2 \), and \( \sigma_3 \) is not a differential ideal. 

A more systematic way of determining the intrinsic torsion is as follows. The Levi-Civita connection is \( w^1\otimes w^{23}-w^2\otimes w^{13}+w^3\otimes w^{12} \), and the intrinsic torsion is minus its projection to \( T^*_mM\otimes (\lie{sp}(2)+\lie{sp}(1))^\perp \). Composing with the skew-symmetrisation map \( \mathbf{a}\colon T^*_mM\otimes \so(8)\to \Lambda^3T_m^*M \), we obtain 
\begin{equation*}
 \tfrac3{20}(w^1\w \sigma_1 + w^2\w \sigma_2 + w^3\w\sigma_3)  + \gamma, \, \gamma\in KH.
\end{equation*}
The second term is not zero, and the first term lies in a submodule isomorphic to \( EH \), namely the image of the equivariant map 
\begin{equation*}
  EH\to \Lambda^3T^*M, \, v\mapsto  (v\hook\sigma_s)\w \sigma_s.
\end{equation*}

Now observe that the kernel of \( \mathbf{a} \) intersects \( T^*_mM\otimes (\spto))^\perp \) in \( KS^3H \), so this determines all of the intrinsic torsion except possibly for \( KS^3H \). Since we have already observed that the structure is quaternionic, this component is forced to be zero. Summing up, the only non-trivial components in the intrinsic torsion are \( KH \) and \( EH \).

\subsubsection{The algebra \( \h_3 \)}
Consider the case when \( (p,q,r)=(1,0,0) \). Then we have the structural equations \( de^i=0 \), \( dw^s = \omega_s \), giving the quaternionic Heisenberg group \( \LH_7 \). In this case, the intrinsic torsion is \( (\lambda,\mu)=(-1/2,1) \) which implies
\( d\alpha=3\beta+3\upsilon\), \( d(\Hodge\upsilon)=-\beta-3\upsilon \).

The product \( M=S^1\times \LH_7 \) is not \HKT. It is, however, a so-called \QKT manifold (see \cite[Definition 7.1]{Cabrera-S:TheIntrinsicTorsion}) and interestingly appears in \cite{Fino-G:Propertiesofmanifolds}. 

The same arguments as used for \( \h_2 \) show that the non-vanishing components of the intrinsic torsion are \( KH \) and \( EH \).

\subsubsection{The algebra \( \h_4 \)}
If \( (p,q,r)=(0,-2,1) \), we obtain the Fino-Tomassini example \( \LH_4=\Sp(1)\ltimes \bR^4 \) (cf. \cite{Fino-Tomassini:GeneralizedG2}), meaning the Lie group 
whose Lie algebra is given by \( de^i=e_i\hook\omega_s\w w^s \), \( dw^s = -2w_s\hook w^{123} \). Note that we may view \( \LH_4 \) as \( \Sp(1)\ltimes \mathbb H \) 
by defining multiplication by \( i,j,k \) via the adjoint action of \( w_1, w_2 \), and \( w_3 \). The manifold \( \LH \), notably, has a compact quotient. 

Again, we can compute the exterior derivatives of \( \alpha \) and \( \Hodge\upsilon \)
by using \( (\lambda,\mu)= (0,-1) \). We find \(  d\alpha=-2\beta-6\upsilon \) and \( d(\Hodge\upsilon)=0 \).

The same arguments we have used before apply to the product \linebreak \( \Sp(2)\Sp(1) \)-structure and show that it has non-zero intrinsic torsion in both components of \( EH\oplus KH \). In particular, we have provided another example of a \QKT structure which is not \HKT.

\begin{remark}
The equations used to construct invariant intrinsic torsion structures on \( \h_2, \h_3,\h_4 \) also apply to give homogeneous examples that are not Lie groups. For instance, the \( 7 \)-dimensional sphere, \( S^7 \), viewed as the homogeneous space \( \Sp(2)\Sp(1)\slash\Sp(1)\Sp(1) \) (cf. \cite[Example 5]{Conti:Intrinsictorsionin}), admits a natural \( \SO(4) \)-structure with invariant intrinsic torsion. This corresponds to \( (p,q,r)=(1,0,1) \).  
\end{remark}

The remaining classes of examples, appearing in Theorem \ref{thm:invintrtor}, are closely tied to self-dual Einstein Lie groups. In order to arrive at these classes, one can impose suitable conditions on the torsion of the flat connection in order to ensure \( \h \) has structural equations given by
\begin{equation*}
\begin{gathered}
de^i=-\phi(w_s)\w e_i\hook\omega_s - \psi(w_s)\w  e_i\hook\varpi_s,\\
dw^s=p\omega_s+q w_s\hook w^{123}-\phi(w_t)\w   w_s\hook (-2w_t\hook w^{123}). 
\end{gathered}
\end{equation*}
In the above, \( \phi,\psi\colon\, \bR^3\to\bR^4 \) are linear maps representing the component of the torsion in \( \partial((\bR^4)^*\otimes\so(4)) \),  and as before we identify vectors and covectors via \( e_i\mapsto e^i \), \( w_s\mapsto w^s \). These equations define a Lie algebra if and only if \( d^2=0 \); then the standard coframe defines an \( \SO(4) \)-structure with invariant intrinsic torsion.

In this case, we can view \( \Span{e^1,\dotsc,e^4} \) as the dual of a Lie algebra \( \lk \), and the projection \( \h\to\lk \) is a Lie algebra homomorphism. The exterior derivative is required to satisfy
\begin{equation*}
\begin{gathered}
0= d^2w^1\\
=pq(\omega_2\w w^3-\omega_3\w w^2)+4\phi(w_1) \w(\phi(w_2)\w    w^2  +\phi(w_3)\w   w^3 )\\
-\phi(w_s)\w  ( -2\phi(w_2)\hook\omega_s \w w^3 +2\phi(w_3)\hook\omega_s \w   w^2),\\
- \psi(w_s)\w ( -2\phi(w_2)\hook\varpi_s \w w^3 +2\phi(w_3)\hook\varpi_s \w   w^2),\\
\end{gathered}
\end{equation*}
which, by considering the cyclic permutations, leads to the relation:
\begin{equation}
\label{eqn:flambdamu1}
 \phi(w_t)\w  \phi(w_s)\hook\omega_t + \psi(w_t)\w \phi(w_s)\hook\varpi_t =\langle \phi\w\phi, w_s\hook w^{123}\rangle-\tfrac12pq\omega_s.
\end{equation} 
Thus, \( \h \) is a Lie algebra if and only if \( \lk \) is a Lie algebra and \eqref{eqn:flambdamu1} holds.

Decreeing \( e^1,\ldots, e^4 \) to be an oriented orthonormal basis of \( \lk \), we can interpret this equation in terms of the Levi-Civita connection, \( -\phi(w_s)\otimes \omega_s -\psi(w_s)\otimes \varpi_s \), on \( \lk \). Indeed, its associated curvature is given by 
\begin{equation*}
\begin{gathered}
R=-d\phi(w_s)\otimes \omega_s -d\psi(w_s)\otimes \varpi_s
+ \tfrac12\phi(w_s)\w \phi(w_t)\otimes [\omega_s,\omega_t]\\
+ \tfrac12\psi(w_s)\w \psi(w_t)\otimes [\varpi_s,\varpi_t]
\end{gathered}
\end{equation*}
which means that \eqref{eqn:flambdamu1} is equivalent to
\begin{equation*}
R= -\tfrac12pq\omega_s\otimes \omega_s-d\psi(w_s)\otimes \varpi_s
+\langle \psi\w\psi, w_t\hook w^{123}\rangle\otimes \varpi_t.
 \end{equation*}
In particular, \( \lk \) is Einstein and self-dual with scalar curvature \( 3pq \) (see \cite{Salamon:TopicsInFour}). 

From the classification  of flat and anti-self-dual metrics on \( 4 \)-dimensional Lie groups in \cite{Milnor:Curvaturesofleft} and \cite{De-Smedt-S:Anti-self-dual}, we deduce that there are essentially three non-trivial possibilities.

\subsubsection{The families \( \h_5^a,  \h_6^a \)}
Assume first that \( \lk \) is flat, but not Abelian (otherwise we would recover the examples \( \h_2 \) and \( \h_3 \)). Then, by \cite{Milnor:Curvaturesofleft}, we can assume \( \lk = (0,0,14,-13) \) so that
\begin{equation*}
\phi(w_1)=-\psi(w_1) = \tfrac12  e_1,\, \phi(w_2) =0=\psi(w_2)=\phi(w_3)=\psi(w_3).
\end{equation*}
By construction \( pq=0 \); for \( (p,q)=(a,0) \) we obtain \( \h_5^a \), and \( \h_6^{a} \) for \( (p,q)=(0,a) \). In either case, the almost hypercomplex structure defined by the forms \( \sigma_i \) on \( S^1\times \LH \) is not integrable; the non-vanishing components of the intrinsic torsion are \( EH \) and \( KH \), apart from the case \( a=0 \) when \( \h_5^0=\h_6^0 \) is flat and hyper-K\"ahler.

\subsubsection{The family \( \h_7^{a,\kappa} \)}
If \( \h \) is only conformally flat, but not flat, the Einstein condition implies that it has constant sectional curvature. The \( 4 \)-sphere has no Lie group structure, so we are dealing with hyperbolic space, which has a one-parameter family of Lie group structures (cf. \cite{Jensen:HomogeneousEinsteinSpaces}) corresponding to \( \lk = (0,12,13+\kappa14,14-\kappa13) \). We then compute
\begin{equation*}
\begin{gathered}
 \phi(w_1) =  \tfrac{\kappa}2e_1 -\tfrac12 e_2,\, \psi(w_1) = - \tfrac{\kappa}2 e_1 -\tfrac12 e_2,\\
\phi(w_2)=-\tfrac12 e_3 = \psi(w_2),\, \phi(w_3) =  -\tfrac12 e_4 = \psi(w_3),
\end{gathered}
\end{equation*}
and \( pq=-1 \). The corresponding \( 8 \)-manifold \( S^1\times \LH \) is only hypercomplex when \( \kappa=0 \); regardless of \( \kappa \), the non-zero intrinsic torsion is in \( EH\oplus KH \).

\subsubsection{The family \( \h_8^{a,\kappa} \)}
If \( \lk \) is not conformally flat,  based on the study of \cite{De-Smedt-S:Anti-self-dual},  we deduce that \( \lk \) is a solvable algebra belonging to the family: 
\begin{equation*}
 \lk=(0,f^{12}+\kappa f^{13}, -\kappa f^{12}+f^{13},2f^{14}+f^{23}). 
 \end{equation*}

A multiple of \( (f^1,\tfrac12f^2,\tfrac12f^3,\tfrac12f^4) \) forms an orthonormal basis. In other words, up to scale,  we have an orthonormal basis \( e^i \) such that 
\begin{equation*}
\lk=(0,e^{12}+\kappa e^{13}, -\kappa e^{12}+e^{13},2e^{14}+2e^{23}).
\end{equation*}

Considering the connection form of \( \lk \), we deduce that
\begin{equation*}
\begin{gathered}
\phi(w_1)=-e_2,\,\phi(w_2)=-e_3,\,\phi(w_3)=-\tfrac{1}{2} e_4+\tfrac{\kappa}2  e_1,\\
\psi(w_1)=0=\psi(w_2),\,\psi(w_3)=-\tfrac32 e_4-\tfrac{ \kappa}2  e_1.
\end{gathered}
\end{equation*}
So in this case, the relation~\eqref{eqn:flambdamu1} is satisfied with \( pq=-2 \), and the curvature form of \( \h \) is given by 
\begin{equation*}
  \omega_s\otimes \omega_s  +3\varpi_3\otimes\varpi_3. 
\end{equation*}

In summary, we have obtained a \( 2 \)-parameter family of \( 7 \)-dimensional Lie algebras with a constant intrinsic torsion \( \SO(4) \)-structure; actually, there are three parameters, namely \( p,q,\kappa \), but these are subject to the constraint \( pq=-2 \).

The \( \Sp(2) \)-structure, induced on the \( 8 \)-manifold \( S^1\times \LH \) via the \( 2 \)-forms \( \sigma_i \), does not correspond to an integrable hypercomplex structure nor satisfies the differential ideal condition. In terms of the intrinsic torsion, we find that  the only non-zero components are \( EH \) and \( KH \). 

\bigskip

We emphasise that the only Einstein examples obtained above are the flat ones, i.e., the structures on \( \h_1 \) and \( \h_5^0=\h_6^0 \). Similarly, one can check directly that among the \( \Sp(2)\Sp(1) \)-structures constructed in this section, the only ones that satisfy the differential ideal condition are those corresponding to these two, consistently with the intrinsic torsion calculations.

This concludes the proof of Theorem~\ref{thm:invintrtor}.

\qed

\bibliographystyle{plain}

\begin{thebibliography}{25}

\bibitem[Be]{Besse:EinsteinManifolds}
A.~L. Besse,
\newblock {\it Einstein manifolds},
\newblock Classics in Mathematics. Springer-Verlag, Berlin (2008),
\newblock (Reprint of the 1987 edition).

\bibitem[Bo]{Bock:OnLowDimensional}
C.~Bock,
\newblock {\it On low-dimensional solvmanifolds},
\newblock arXiv:0903.2926 [math.DG].


\bibitem[Br]{Bryant:CalibratedEmbeddingsIn}
R.~L. Bryant,
\newblock {\it Calibrated embeddings in the special {L}agrangian and coassociative
  cases},
\newblock Ann. Global Anal. Geom., {\bf 18 } (2000), no. 3--4, 405--435.

\bibitem[Ca]{Catellani:TriLagrangianmanifolds}
G.~Catellani.
\newblock \emph{Tri-Lagrangian manifolds},
\newblock PhD thesis, Universit{\`{a}} degli studi di Firenze (2003).

\bibitem[CM]{Chiossi-M:SO3StructuresOn}
S.~G. Chiossi and {\'O}. Maci{\'a},
\newblock {\it \( {\rm SO}(3) \)-structures on 8-manifolds.}
\newblock Ann. Global Anal. Geom., {\bf 43} (2013), no. 1, 1--18.

\bibitem[Co1]{Conti:Halfflatnilmanifolds}
D.~Conti,
\newblock {\it Half-flat nilmanifolds},
\newblock Math. Ann., {\bf 350} (2011), no. 1, 155--168.

\bibitem[Co2]{Conti:Intrinsictorsionin}
D.~Conti,
\newblock {\it Intrinsic torsion in quaternionic contact geometry},
\newblock arXiv:1306.0890 [math.DG].

\bibitem[CF]{Conti-F:Nilmanifoldswithcalibrated}
D.~Conti and M.~Fern{\'a}ndez,
\newblock {\it Nilmanifolds with a calibrated { \( \G_2 \)}-structure},
\newblock Differential Geom. Appl., {\bf 29} (2011), 493--506.

\bibitem[DS]{De-Smedt-S:Anti-self-dual}
V.~De~Smedt and S.~Salamon,
\newblock {\it Anti-self-dual metrics on {L}ie groups},
\newblock Contemp. Math., {\bf 308} (2002), pages 63--75.

\bibitem[FG]{Fino-G:Propertiesofmanifolds}
A.~Fino and G.~Grantcharov,
\newblock {\it Properties of manifolds with skew-symmetric torsion and special holonomy},
\newblock Adv. Math., {\bf 189} (2004), no. 2, 439--450.

\bibitem[FT]{Fino-Tomassini:GeneralizedG2}
A.~Fino and A.~Tomassini,
\newblock {\it Generalized {\( \G_2 \)}-manifolds and {\( {\rm SU}(3) \)}-structures},
\newblock Internat. J. Math., {\bf 19} (2008), no. 10, 1147--1165.

\bibitem[Gi]{Giovannini:SpecialStructuresAnd}
D.~Giovannini,
\newblock {\it Special structures and symplectic geometry},
\newblock PhD thesis, Universit{\`{a}} degli studi di {T}orino (2003).

\bibitem[Go]{Gong:Classificationofnilpotent}
M.-P. Gong,
\newblock {\it Classification of nilpotent {L}ie algebras of dimension \( 7 \)
  (over algebraically closed fields and \( \bR \))},
\newblock PhD Thesis, University of Waterloo (1998).

\bibitem[Gr]{Gray:ANoteOn}
A.~Gray,
\newblock {\it A note on manifolds whose holonomy group is a subgroup of \( \Sp(n)\Sp(1) \)},
\newblock Michigan Math. J., {\bf 16} (1969), 125--128.

\bibitem[H1]{Htichin:TheSelfDuality}
N.~Hitchin,
\newblock {\it The self-duality equations on a {R}iemann surface},
\newblock Proc. London Math. Soc., {\bf 55} (1987), no. 1, 59--126.

\bibitem[H2]{Hitchin:Stableformsand}
N.~Hitchin,
\newblock {\it Stable forms and special metrics},
\newblock Contemp. Math., {\bf 288} (2001), 70--89. 

\bibitem[Je]{Jensen:HomogeneousEinsteinSpaces}
G.~R. Jensen,
\newblock {\it Homogeneous {E}instein spaces of dimension four},
\newblock J. Differential Geometry, {\bf 3} (1969), 309--349.

\bibitem[Jo]{Joyce:CompactManifoldsWith}
D.~D. Joyce,
\newblock {\it Compact manifolds with special holonomy},
\newblock Oxford Mathematical Monographs. Oxford University Press, Oxford (2000).

\bibitem[LS]{LeBrun-S:Strongrigidityof}
C.~LeBrun and S.~Salamon,
\newblock {\it Strong rigidity of positive quaternion-k\"ahler manifolds},
\newblock Invent. Math., {\bf 118} (1994), no. 1, 109--132.

\bibitem[MS]{Cabrera-S:TheIntrinsicTorsion}
F.~Mart{\'{\i}}n~Cabrera and A.~Swann,
\newblock {\it The intrinsic torsion of almost quaternion-{H}ermitian manifolds.}
\newblock Ann. Inst. Fourier (Grenoble), {\bf 58} (2008), no. 5, 1455--1497.

\bibitem[Mi]{Milnor:Curvaturesofleft}
J.~Milnor,
\newblock {\it Curvatures of left invariant metrics on {L}ie groups},
\newblock Adv. Math., {\bf 21} (1976), no. 3, 293--329.

\bibitem[Sa1]{Salamon:TopicsInFour}
S.~Salamon,
\newblock {\it Topics in four-dimensional {R}iemannian geometry},
\newblock in: Geometry seminar ``{L}uigi {B}ianchi\/'',
  Springer Lecture Notes, {\bf 1022} (1983), pages 33--124.

\bibitem[Sa2]{Salamon:Self-dualityand}
S.~Salamon,
\newblock {\it Self-duality and exceptional geometry},
\newblock (1987), 
\newblock http://calvino.polito.it/\( \sim \) salamon/G/baku.pdf.

\bibitem[Sa3]{Salamon:RiemannianGeometryAnd}
S.~Salamon,
\newblock {\it Riemannian geometry and holonomy groups}, 
\newblock Pitman Research Notes in Mathematics Series (Longman, Harlow, 1989).

\bibitem[Sa4]{Salamon:Almostparallelstructures}
S.~Salamon,
\newblock {\it Almost parallel structures},
\newblock Contemp. Math., {\bf 288} (2001), 162--181.

\bibitem[Sw]{Swann:AspectsSymplectiquesDe}
A.~Swann,
\newblock {\it Aspects symplectiques de la g\'eom\'etrie quaternionique},
\newblock \emph{C. R. Acad. Sci. Paris S\'er. I Math.}, {\bf 308} (1989), no. 7, 225--228.

\bibitem[Wi]{Witt:Specialmetricsand}
F.~Witt,
\newblock {\it Special metrics and triality},
\newblock \emph{Adv. Math.}, {\bf 219} (2008), no. 6, 1972--2005.

\end{thebibliography}

\end{document}